\documentclass[11pt,a4paper]{article}
\usepackage{amsmath,amssymb,amsthm}
\usepackage{psfrag}
\usepackage{graphicx,subfigure,float,url,color}
\usepackage[colorlinks=true]{hyperref}
\usepackage{enumerate}
\usepackage{comment}

\newcommand{\dsp}{\displaystyle}
\newcommand{\R}{\mathbb R}
\newcommand{\N}{\mathbb N}
\newcommand{\C}{\mathbb C}

\newcommand{\xd}{\mathrm{d}}

\newcommand{\B}{\mathcal B}
\newcommand{\Om}{\Omega}
\newcommand{\var}{\operatorname{Var}}
\newcommand{\ee}{\boldsymbol{e}}
\newcommand{\yy}{\boldsymbol{y}}
\newcommand{\vv}{\boldsymbol{v}}

\newcommand{\I}{\mathbb{I}}
\newcommand{\id}{\mathrm{Id}}
\newcommand{\im}{\operatorname{im}}
\renewcommand{\span}{\operatorname{Span}}
\newcommand{\esssup}{\operatorname{ess\,\,sup\,}}
\newcommand{\essinf}{\operatorname{ess\,inf\,}}
\newcommand{\essran}{\operatorname{ess\,ran}}
\newcommand{\esupp}{\mathrm{ess\,supp\,}}
\renewcommand{\geq}{\geqslant}
\renewcommand{\leq}{\leqslant}
\newcommand{\dt}{\Delta t}
\newcommand{\oplor}{\stackrel{\scriptstyle\perp}{\oplus}}
\newcommand{\oplov}{\stackrel{\scriptstyle\perp_v}{\oplus}}
\newcommand{\ie}{\textit{i.e. }}
\newcommand{\re}{\operatorname{Re}}
\newcommand{\eps}{\varepsilon}
\newcommand{\spec}{\mathfrak{S}}
\newcommand{\exv}{\operatorname{E}_v}

\newtheorem{theorem}{Theorem}  
\newtheorem{proposition}[theorem]{Proposition}

\theoremstyle{definition}\newtheorem{remark}{Remark}

\title{Exponential convergence towards consensus \\ for non-symmetric linear first-order systems \\ in finite and infinite dimensions\footnote{This work was partially funded by the ANR-14-ACHN-0030-01 project
\textit{Kimega} and the ANR-20-CE40-0009-04 project \textit{TRECOS}.}}

\author{
Laurent Boudin \footnote{Sorbonne Universit\'e, CNRS, Universit\'e de Paris, Laboratoire Jacques-Louis Lions (LJLL), F-75005 Paris, France (\texttt{laurent.boudin@sorbonne-universite.fr})}
\and
Francesco Salvarani \footnote{L\'eonard de Vinci P\^ole Universitaire, Research Center, 92916 Paris La D\'efense, France \& Dipartimento di Matematica ``F. Casorati'',
Universit\`a degli Studi di Pavia, Via Ferrata 1, 27100 Pavia, Italy (\texttt{francesco.salvarani@unipv.it})}
\and
Emmanuel Tr\'elat \footnote{Sorbonne Universit\'e, CNRS, Universit\'e de Paris, Inria, Laboratoire Jacques-Louis Lions (LJLL), F-75005 Paris, France (\texttt{emmanuel.trelat@sorbonne-universite.fr})}
}

\begin{document}
\maketitle

\begin{abstract}
We consider finite and infinite-dimensional first-order consensus systems with time-constant interaction coefficients. For symmetric coefficients, convergence to consensus is classically established by proving, for instance, that the usual variance is an exponentially decreasing Lyapunov function. We investigate here the convergence to consensus in the non-symmetric case: we identify a positive weight which allows to define a weighted mean corresponding to the consensus, and obtain exponential convergence towards consensus. Moreover, we compute the sharp exponential decay rate. 
\end{abstract}

\section{Introduction} \label{s:intro}

The study of self-organizing dynamics is a very active research subject and a
vast literature is focused on models describing alignment or
agreement/disagreement phenomena caused by the interaction between agents, see
\cite{castellano2009statistical} for instance. 
There exist various strategies, such as the approach based on cellular
automata allowing a continuum of states obeying to a probabilistic interaction 
law \cite{weisbuch2003interacting}, or models describing the time evolution of 
probability densities, see \cite{bou-sal, tos06}.

We consider here the viewpoint of dynamical systems and study
first-order consensus dynamics, following the works of Hegselmann and Krause 
\cite{heg-kra, kra}. We recall moreover that there are other celebrated 
consensus models, for example the alignment second-order models due to 
Vicsek {\it et al.} \cite{vicsek1995novel, vic-zaf-2012} and to Cucker 
and Smale \cite{cuc-sma, cuc-sma-2}, see also \cite{ajm-bel-gib}.

Many studies of the Hegselmann-Krause model were recently implemented, as in \cite{bic-ko-zua, has21, jab-mot, mot-tad, pao21}. We point out \cite{weber2020bounded, web-the-mot}, where the authors proposed graph-theory related ideas to develop their theory. Second-order models were also studied in \cite{degond2008continuum, ha2009simple, jab-mot, mot-tad-2011, shen2007cucker}. There are of course natural questions arising on the control on these models, which led, among other works to \cite{altafini2013consensus,ayd-cap-etal, blondel2009krause, caponigro2013sparse, caponigro2015sparse, olshevsky2009convergence, piccoli2015control2, wongkaew2015control}. The models were also embedded in more sophisticated ones, as in \cite{barbaro2016phase}, or allowed to establish some models hierarchy in asymptotic analysis studies, \textit{e.g.} \cite{carrillo2014derivation, carrillo2010asymptotic, degond2013macroscopic, ha2008particle}. The production of accurate numerical strategies for the quantitative study of these models has been the
goal of, for example, \cite{albi2013binary}.

\paragraph{First-order linear consensus model in finite dimension.}
The Hegselmann-Krause model allows to describe the time evolution of a set of
time-depending state variables $y_i:\R_+\to\R$, $1\leq i \leq N$. Each state
variable evolves with respect to time, according to the following first-order
differential system, which quantifies the modifications induced by the
interactions between all state variables of the system. For any
$1\leq i,j\leq N$, let $\sigma_{ij}\geq 0$ be the interaction frequency of the agent 
$i$ with the agent $j$. The differential system is then written as
\begin{equation} \label{e:krause}
\dot y_i(t) = \sum_{j=1}^N \sigma_{ij} \left( y_j(t) - y_i(t)
\right), \qquad t\in\R_+, \qquad  1\leq i \leq N. 
\end{equation}
The right-hand side of \eqref{e:krause} stands for binary interactions between
agents. Due to the form of the system, the values of $\sigma_{ii}$, $1\leq i \leq
N$, can be arbitrarily chosen since they do not influence the dynamics. Equation~\eqref{e:krause} is supplemented with initial conditions
\begin{equation} \label{e:initconddiscrete}
y_i(0)=y_i^{\mathrm{in}}, \qquad  1\leq i \leq N,
\end{equation}
where $y_i^{\mathrm{in}} \in \R$, $1\leq i \leq N$. 
The linear finite-dimensional Cauchy problem
\eqref{e:krause}--\eqref{e:initconddiscrete} has of course a unique global solution. 

The state variables have many possible interpretations. For example, we can 
consider a population composed of $N$ agents and that, for all $i$, $y_i$ 
represents the position of the agent $i$. In this case,
Equations~\eqref{e:krause}--\eqref{e:initconddiscrete} describe a situation of
\textsl{herding}. More generally, this model belongs to the category of
\textsl{consensus systems} because of its stabilization properties in large
time.

The properties of the system are sensitive with respect to the values of 
$\sigma_{ij}$ as well as the methods of proof. In particular, if the system is 
\textit{symmetric}, \ie when $\sigma_{ij}=\sigma_{ji}$ for any $1\leq i,j\leq N$, or when
$\sigma_{ij}$ only depends on either $i$ or $j$, the mathematical study of
\eqref{e:krause}--\eqref{e:initconddiscrete} can be widely simplified. 
This explains why these assumptions, though restrictive, are popular.

Equation~\eqref{e:krause} can be rewritten in a matrix form. Let us consider 
the matrices $\sigma=(\sigma_{ij})_{1\leq i,j \leq N}\in \R^{N\times N}$ and 
$A\in \R^{N\times N}$ such that
\begin{equation*}
A_{ij}= \sigma_{ij}~~\mbox{ if } i\neq j, \qquad 
A_{ii}=-\sum_{k\neq i} \sigma_{ik}.
\end{equation*}
If we denote by $\operatorname{diag}(z)$ the diagonal matrix where the nonzero coefficients are given by the coordinates of a vector $z\in\R^N$, and if we set $e=(1,\dots,1)^\intercal \in \R^N$, then we can write
\begin{equation}
\label{eq:Adiscr}
A=\sigma-\operatorname{diag}(\sigma e).
\end{equation}
In fact, $A$ can be seen as an arbitrary $N\times N$ real matrix whose
off-diagonal coefficients are nonnegative, and such that the sum of coefficients
of any of its rows is zero. With these notations, the linear problem
\eqref{e:krause}--\eqref{e:initconddiscrete} is written, with the state
$$
y:\R_+\to\R^N, ~t\mapsto 
\begin{pmatrix} y_1(t) \\ \vdots \\ y_N(t) \end{pmatrix},$$
 as
\begin{equation} \label{e:pbdimfinie}
\dot y(t) = A y(t), \qquad y(0)= 
\begin{pmatrix} 
y_1^{\mathrm{in}} \\ \vdots \\ y_N^{\mathrm{in}}
\end{pmatrix}:=y^{\mathrm{in}}.
\end{equation}
This matrix form allows to highlight the difficulties induced by the
non-symmetry of $A$ when investigating the large-time behavior of solutions
to \eqref{e:pbdimfinie}. 

\paragraph{Why the usual ``$\boldsymbol{L^2}$ theory'' cannot be applied in the non-symmetric case.}

The proof of convergence towards consensus is easy to establish in the symmetric case by using the standard Euclidean approach, denoting by $\langle\cdot,\cdot\rangle$ the usual Euclidean scalar product in $\R^N$ and by $|\cdot|$ the associated Euclidean
norm. The classical way to prove convergence towards equilibrium of a dynamical
system consists in studying the time decay of $|y-y_{\mbox{\tiny eq}}|$,
where $y_{\mbox{\tiny eq}} \in\ker A$ is the expected consensus. 

Let us first find out the value of $y_{\mbox{\tiny eq}}$. We shall prove below  that $\ker A$ is the one-dimensional space spanned by $e=(1,\dots,1)^\intercal \in \R^N$ (see Proposition\,\ref{t:poids}). When $A$ is self-adjoint, it is straightforward to prove that $t\mapsto \sum_i y_i(t)$ is constant. Indeed, its derivative equals 
$\langle\dot y(t), e\rangle=\langle Ay(t), e\rangle= \langle y(t),A e\rangle=0$.
Therefore $y_{\mbox{\tiny eq}}=\frac1N(\sum_iy_i^{\mathrm{in}}) e$.

Besides, one has 
\begin{equation*}
\frac12 \frac\xd{\xd t} |y(t)-y_{\mbox{\tiny eq}}|^2 = 
\langle A(y(t)-y_{\mbox{\tiny eq}}),y(t)-y_{\mbox{\tiny eq}}\rangle.
\end{equation*}
Since, thanks to the Gershgorin circle theorem, for any eigenvalue $\mu$ of $A$, there exists $i$ such that
\begin{equation} \label{e:gershgorin}
\Big|\mu+\sum_j \sigma_{ij}\Big| \leq \sum_j \sigma_{ij},
\end{equation}
it is clear that, apart from $0$, all eigenvalues of $A$ have a negative real part. Consequently, if $A$ is self-adjoint, then the system is dissipative and we can prove the convergence of $y$ towards $y_{\mbox{\tiny eq}}$ in large time, with an explicit exponential rate given in terms of nonzero eigenvalues of $A$. 

In contrast, the quantity $\langle Ay,y\rangle$ may be positive for some
$y\in\R^N$ when $A$ is not self-adjoint. For instance, if we choose, for
any $j$, $\sigma_{ij}=1$ when $i<N$ and $\sigma_{Nj}=\sigma_N\neq1$, it is
easy to check that $\langle Ay,y\rangle>0$, with
$y=(1,\cdots,1,\gamma)^\intercal$, for some values of $\gamma\in\R$. It is of
course not contradictory with the the fact that all eigenvalues of $A$ have
nonpositive real parts. 

Therefore, the study of the non-symmetric linear Hegselmann-Krause model cannot
rely on a standard variance-based strategy. In \cite{jab-mot, mot-tad-2011}, the authors
have developed a \emph{$L^\infty$ theory}, or hybrid theories which are at least
partially based on $L^\infty$ tools. Although the latter approach is remarkable, the 
$L^2$ framework remains more convenient to investigate stability and robustness properties of the system, or to design asymptotically stabilizing controls, see \cite{weber2020bounded, web-the-mot}. 

\paragraph{First-order consensus model in infinite dimension.}

In this work, we also tackle the problem generalized to set of a continuum of agents. The corresponding model describes the time evolution of a continuum of agents labelled by a continuous variable lying in an open bounded subset $\Om\subset \R^d$, $d\geq 1$. Without loss of generality, we assume that $|\Om|=1$ (Lebesgue measure of $\Om$). The space $\R^d$ is endowed with the standard Euclidean norm $|\cdot|$. 

Let $\sigma\in L^\infty(\Om^2)$ the generalized nonnegative interaction function, and define $S\in L^\infty(\Om)$ related to $\sigma$ through
\begin{equation} \label{e:defS}
S(x)=\int_{\Om} \sigma(x,x_*)\, \xd x_*, \qquad \mbox{for a.e.}~x\in\Om.
\end{equation}
The unknown is now a function $y:\Om\times\R_+ \to \R$
which evolves according to the dynamics 
\begin{equation}\label{e:diminfinie}
\frac{\partial y}{\partial t}(x,t) = 
\int_\Om \sigma(x,x_*) ( y(x_*,t)-y(x,t) ) \, \xd x_*,
\end{equation}
for $x\in\Om$ and $t\geq 0$, with initial condition 
\begin{equation}\label{e:diminfinieinitcond}
y(\cdot,0) = y^{\mathrm{in}},
\end{equation}
where $y^{\mathrm{in}}\in L^2(\Om)$ is given. 

The previously defined interaction matrix $A$ becomes a linear 
interaction operator $A:L^2(\Om)\to L^2(\Om)$ defined, for every 
$z\in L^2(\Om)$, by
\begin{equation}\label{d:operateurA}
(Az)(x) = \int_\Om \sigma(x,x_*) ( z(x_*)-z(x) ) \, \xd x_*, \qquad
\mbox{for a.e.}~x\in\Om.
\end{equation}
Since $\sigma$ lies in $L^\infty$ and $\Om$ is bounded, $A$ is
bounded. Eventually, we emphasize that $\sigma$ is not assumed to have any 
symmetry properties, so that $A$ is not self-adjoint in general.

The operator $A$ only acts on the labelling variable $x\in\Om$. Hence, the Cauchy problem \eqref{e:diminfinie}--\eqref{e:diminfinieinitcond} can be written as
\begin{equation} \label{e:pbdiminfinie}
\frac{\partial y}{\partial t} = A y, \qquad y(\cdot,0) = y^{\mathrm{in}}.
\end{equation}
Since $A$ is bounded, \eqref{e:pbdiminfinie} has a unique global solution $t\mapsto e^{tA}y^{\mathrm{in}}$, which lies in $C^1(\R_+;L^2(\Om))$. We also have a maximum principle on $y$. To recover this result, we first need to introduce the (Hilbert-Schmidt hence compact) operator $K:L^2(\Om)\to L^2(\Om)$ of kernel $\sigma$ defined by
\begin{equation}\label{e:defK}
(K z)(x) = \int_\Om \sigma(x,x_*) z(x_*) \, \xd x_*, \qquad
\mbox{for a.e.}~x\in\Om,
\end{equation}
and the (bounded) multiplication operator $M_S:L^2(\Om)\to L^2(\Om)$ by $S$, \ie $M_S=S\,\id$ where $\id$ is the identity on $L^2(\Om)$. Then it is possible to write \eqref{d:operateurA} under the form,  to be related to \eqref{eq:Adiscr},
\begin{equation*}
A = K - M_S = K-S\,\id= K-(Ke)\,\id,
\end{equation*}
where $e$ is the constant function equal to $1$, and \eqref{e:diminfinie} as
$$\frac{\partial y}{\partial t}+M_S y = Ky.$$
Using $(x,t)\mapsto y(x,t)e^{S(x)t}$ and the fact that $K$ obviously
preserves nonnegativity, we conclude that the solution of \eqref{e:diminfinie}
remains nonnegative almost everywhere if its initial datum is nonnegative, and
that $y$ remains bounded between the essential infimum and supremum of
$y^{\mathrm{in}}$. 

The article is structured as follows. In the next section, we state our theorems of convergence to consensus. The proof of those results starts with preliminary results, mainly of geometric nature, in Section~\ref{s:AAstarv} and is concluded in Section~\ref{s:convcons}. Then, in Section~\ref{s:lyap}, we present further results: the time-discrete version of the finite-dimensional model, the link between our models, noting that the finite-dimensional model can be seen as the infinite-dimensional one where the counting measure is used instead of the Lebesgue measure in the integrals, and some arguments related to Lyapunov functionals. Eventually, in Section~\ref{s:num}, we describe some numerical simulations.

\section{Main results} \label{s:main}

In what follows, $X$ will denote the state space, which is a Hilbert space endowed with its scalar product $\langle\cdot,\cdot\rangle$ and norm $\|\cdot\|$. This space will be either $\R^N$ or $L^2(\Om)$ endowed with their usual scalar product. With both finite and infinite dimension notations, we have $Ae=0$. Besides, it will also be convenient to denote by $W$ the Banach space where $\sigma$ lies, either $\R^{N\times N}$ or $L^\infty(\Om^2)$.

\subsection{Strong connectivity}
Our work relies on graph-theory-related assumptions on $\sigma$ already discussed in \cite{web-the-mot} when $X=\R^N$. 

\subsubsection*{In finite dimension}
We associate to $(\sigma_{ij})$ the directed graph $G$ (see for instance \cite[Chapter~10]{bon-mur}), whose vertices are $1$, $2$, $\ldots$, $N$, and which has an edge from $i$ to $j$ when $\sigma_{ij}>0$. This concept of directed graph allows to handle the heterogeneity of the reciprocal influence of the agents. In this setting, the agents are the vertices and the matrix $A$ is linked to the edges between two given vertices. More precisely, when an entry of $A$ is zero, there is no direct interaction between the corresponding agents and when an entry of $A$ is positive, the corresponding agents are directly connected. We recall that $G$ is \emph{strongly connected} if, for any pair $(i,j)$ with $i\neq j$, there exists a a finite set of arcs, called a path, joining $i$ to $j$ in $G$, \ie there exists a sequence $(i_0, \dots, i_r)$, $r\in\N^*$, of distinct indices satisfying
$$i_0=i, \quad i_r=j, \qquad \sigma_{i_{k}i_{k+1}}>0, \quad 0\leq k\leq r-1.$$ 
Usually, the strong connectivity definition states that, for any pair $(i,j)$ with $i\neq j$, there exist a path joining $i$ to $j$ \textit{and} a path joining $j$ to $i$, to differ from the weak connectivity notion, for which the previous \textit{and} is replaced by \textit{or}. 

The strong connectivity assumption can be interpreted as a constraint on the links between the agents of the system: for instance, when $G$ is strongly connected, any pair $(i,j)$ of individuals of the population can interact directly (if $\sigma_{ij}>0)$, or indirectly through other individuals (when $\sigma_{ij}=0$ but there is a path between $i$ and $j$).

\subsubsection*{In infinite dimension}
The same notion of directed graph is extended in the infinite-dimensional setting in the following way. 

The vertices of the directed graph $G$ associated to $\sigma\in L^\infty(\Om^2)$ are chosen as the Lebesgue points $x$ of $\sigma$ in $\Om$, \ie the ones such that $x_*\mapsto \sigma(x,x_*)$ is defined almost everywhere in $\Om$. Then, for any vertices $x_1$, $x_2$ such that $x_1\neq x_2$, we say that $(x_1,x_2)$ is an arc if $x_2 \in \esupp \sigma(x_1,\cdot)$. 

The directed graph $G$ is strongly connected if both following properties hold: 
\begin{enumerate}
\item For any Lebesgue points $(x,x_*)$ with $x\neq x_*$, there exists a path joining $x$ to $x_*$ in $G$, \ie there exist two-by-two distinct Lebesgue points $x_0$, ..., $x_r$, $r\in\N^*$ such that 
$$x_0=x, \quad x_r=x_*, \qquad x_{k+1} \in \esupp \sigma(x_k,\cdot), \quad 0\leq k\leq r-1.$$ 
\item Recalling that $S$ is defined by \eqref{e:defS}, we have 
\begin{equation} \label{e:hypSdelta}
\delta := \essinf S >0,
\end{equation}
\end{enumerate}
The first property is exactly the one which defined a strongly connected directed graph in finite dimension. The second one, which means that (almost) every agent can interact with a significant continuum of agents in $\Om$, measured thanks to $\delta$, is clearly satisfied in finite dimension, without further assumption. Indeed, for any $i$, there necessarily exists $j\neq i$ such that $\sigma_{ij}>0$. This ensures that the term $\sum_j \sigma_{ij}$ corresponding to $S$ is positive for any $i$. In fact, that second property is also directly satisfied if, for instance, $\sigma$ is continuous on $\bar\Om^2$ and satisfies the first property. 

\subsection{Main results}
Under this strong connectivity assumption, we can now identify the consensus value in terms of an eigenvector $v$ of $A^*$, for which we prove the existence and positivity properties that were only assumed in \cite{olsab-mur}. We also recover the $L^2$-convergence towards consensus obtained in finite dimension in \cite{web-the-mot}, extend it to the infinite-dimensional model, and provide the sharp convergence rate in both cases. 

We state hereafter our two main results, valid in finite and infinite dimensions, and we start by providing the proper consensus value. 
\begin{theorem} \label{t:poids}
Assume that the graph associated to $\sigma$ is strongly connected. Then there exists a unique $v\in\ker A^*$ such that $v>0$ and $\langle v,e\rangle =1$, and the weighted mean of any solution $y$ to \eqref{e:pbdimfinie} or \eqref{e:pbdiminfinie} defined by $\bar y^v=\langle y(t),v\rangle\,e$ is constant with respect to time.
\end{theorem}
Note that, if $\sigma$ is symmetric, then $A$ is self-adjoint, and $v=e/\|e\|^2$. Let us also point out that $v$ was also defined in \cite{olsab-mur}, but the fact that $v>0$ was only assumed, not proved. The weighted mean $\bar y^v$ is the value of the consensus to which any solution to \eqref{e:pbdimfinie} or \eqref{e:pbdiminfinie} converges in large time. 
\begin{theorem} \label{t:convcons}
Assume that the directed graph associated to $\sigma$ is strongly connected. Let $y:\R_+\to X$ solving \eqref{e:pbdimfinie} or \eqref{e:pbdiminfinie}. Then there exists $\rho>0$ such that, for any $\eps\in (0,\rho)$, there exists $M_\eps>0$ satisfying
$$\|y(t)-\bar y^v\| \leq M_\eps \|y^{\mathrm{in}}-\bar y^v\| e^{(-\rho+\eps)t}, \qquad \forall t\geq0.$$
\end{theorem}
The convergence rate $\rho$ was already exhibited in \cite{web-the-mot} when $X=\R^N$, as being $|\re \lambda_2|$, where $\lambda_2$ is the eigenvalue of $A$ whose real part is the highest one, apart from $0$. Note that, if $A$ is a symmetric matrix, $|\re \lambda_2|$ is the so-called Fiedler number, which measures the strong connectivity of the graph associated to $(\sigma_{ij})$, see \cite{olsab-mur}. We shall prove that the sharp value of $\rho$ is
$$\rho=\mathrm{s}(A_2),$$
where $A_2:\im A \to \im A$ is the homoeomorphism defined by $A_2 z=Az$ for every $z\in \im A$ (indeed, we shall see that $\im A = \im A^2$), and $\mathrm{s}(A_2)<0$ is the spectral bound of $A_2$.

\begin{remark} \label{r:partialconnect}
If the strong connectivity property is not satisfied on the whole set of agents, but only on subsets of a disjoint partition of that set, we recover a clustering effect and the optimal exponential convergence rate can be computed in the same way as explained above.
\end{remark}

\paragraph{Strategy of proof.}

Let us briefly explain the key arguments of our proof of Theorems~\ref{t:poids}--\ref{t:convcons}. We start by checking in Proposition~\ref{p:ptykerA-kerAstar} that the null spaces of $A$ and $A^*$ are one-dimensional, and that $\ker A$ is generated by $e$, which is defined by \eqref{e:def-e}. Then we prove that any nonzero element $y\in\ker A^*$ is either positive or negative (meaning that $y>0$ or $y<0$ almost everywhere when $X=L^2(\Om)$ or that all coordinates of $y$ are nonzero and have the same sign when $X=\R^N$). To prove this fact, we perform a deformation of the model: we define a homotopy path joigning the non-symmetric interaction function (or matrix) $\sigma$ to a symmetric one, and we show that, along this path, elements of $\ker A^*$ always keep the same sign. This ensures the existence and uniqueness of a positive weight $v\in\ker A^*$ as in Proposition~\ref{p:poidsrappel}, which is used to compute the consensus value $\bar y^v$. Moreover, we define a Hilbert structure on $X$, equivalent to the standard one, involving a scalar product and a norm weighted by $v$. Proposition~\ref{p:directsums} then clarifies the geometric context of our convergence result. Indeed, in order to prove Theorem~\ref{t:convcons}, we also need to introduce the $v$-orthogonal projector $\pi$ on $\im A$ and the homeomorphism $A_2:\im A\to\im A$, $z\mapsto Az$. In finite dimension, the matrix $A_2$ is Hurwitz, which is enough to conclude, and incidentally recover the well-known property related to the Fiedler number when $A$ is symmetric. In infinite dimension, we perform a thorough study of the discrete and essential spectra of the operators $A$ and $A_2$, beginning with Proposition~\ref{p:spectreA}, where the Hilbert structure weighted by $v$ is crucial, since it allows to work in the appropriate setting, with respect to $v$. The sharp exponential rate is then obtained as $|\mathrm{s}(A_2)|$, where $\mathrm{s}(A_2)<0$ is the spectral bound of $A_2$. 

\section{Properties of $\boldsymbol{A}$ and $\boldsymbol{A^*}$, definition of the weight} \label{s:AAstarv}
This section is dedicated to the preliminary results required to prove Theorems~\ref{t:poids}--\ref{t:convcons}. The reader can focus on the statement (and skip the proof) of the various propositions below, before getting to Section~\ref{s:convcons}. 

Let us first write the expression of $A^*$, which is the same notation for the transposed matrix of $A$ when $X=\R^N$ and the adjoint operator of $A$ when $X=L^2(\Om)$. We have, for any $z\in\R^N$, 
\begin{equation*}
(A^*z)_i = \sum_j \sigma_{ji} z_j- \Big(\sum_j \sigma_{ij}\Big)z_i,\qquad
1\leq i\leq N,
\end{equation*}
and, for any $z\in L^2(\Om)$,
\begin{multline*}
A^*z(x) = \int_\Om \sigma(x_*,x) z(x_*)\,\xd x_*- \left(\int_\Om \sigma(x,x_*)\,\xd x_*\right)z(x) \\
=\int_\Om \sigma(x_*,x) z(x_*)\,\xd x_*- S(x)z(x),\qquad
\mbox{for a.e.}~x\in\Om.
\end{multline*}
Consequently, if $z\in\ker A^*$, we have either, when $X=\R^N$, 
\begin{equation*}
\Big(\sum_j \sigma_{ij}\Big)z_i = \sum_j \sigma_{ji} z_j, \qquad 1\leq i\leq N, 
\end{equation*}
or, when $X=L^2(\Om)$, 
\begin{equation}\label{e:ptySker}
S(x) z(x) = \int_\Om \sigma(x_*,x) z(x_*)\,\xd x_*, \qquad \mbox{for a.e.}~x\in\Om. 
\end{equation}
Let us emphasize that the previous equality implies that $\ker A^*$ is continuously embedded in $L^\infty(\Om)$. Indeed, consider $z\in \ker A^*$. From \eqref{e:hypSdelta}--\eqref{e:ptySker}, we have, for almost every $x\in\Om$, 
$$|z(x)|=\frac{1}{S(x)} \left|\int_\Om \sigma(x_*,x) z(x_*)\,\xd x_*\right| \leq
\frac{\|\sigma\|_{L^\infty}}{\delta} \|z\|_{L^2(\Om)}.$$

We can now proceed with the preliminary propositions leading to the proof of Theorems~\ref{t:poids}--\ref{t:convcons}. We start by recalling that
\begin{equation} \label{e:def-e}
e=(1,\dots,1)^\intercal ~~ \mbox{if }X=\R^N, \qquad e=1 ~~ \mbox{if }X=L^2(\Om).
\end{equation}

\subsection{Null spaces of $\boldsymbol{A}$ and $\boldsymbol{A^*}$}

\begin{proposition} \label{p:ptykerA-kerAstar}
The following properties hold: 
\begin{enumerate}[(i)]
\item $\ker A=\ker A^2$ is a one-dimensional subspace of $X$ spanned by $e$,
\item $\ker A^*=\ker (A^*)^2$ is a one-dimensional subspace of $X$, 
\item $0$ is a simple eigenvalue of both $A$ and $A^*$.
\end{enumerate}
\end{proposition}
Note that we can also obtain that $\dim \ker A=1$ thanks to the Perron-Frobenius theorem, as in \cite{web-the-mot}.

\begin{proof} Proposition~\ref{p:ptykerA-kerAstar} is proved hereafter in a unified way, not depending on the fact that we work in a finite-dimensional setting or not. However, its proof may require some arguments based on dimension-related arguments on $X$.

\subsubsection*{Step 1 -- $\boldsymbol{\ker A}$ and $\boldsymbol{\ker A^*}$ are finite-dimensional subspaces of $\boldsymbol{X}$.}
Since it is obvious when $X=\R^N$, let us focus on the infinite-dimensional setting. The fact that $\sigma \in L^\infty(\Om^2)$ also implies that the operator $K$ given by \eqref{e:defK} is compact $L^2(\Om)\to L^2(\Om)$ as a Hilbert-Schmidt operator with kernel $\sigma \in L^2(\Om^2)$. Consequently, thanks to the assumptions on $S$, $M_S^{-1}K$ is compact $L^2(\Om)\to L^2(\Om)$. Hence, $M_S^{-1}K - \id$ satisfies the Fredholm 
alternative. In particular,
\begin{equation} \label{e:fredhMSKs}
\dim \ker (M_S^{-1}K-\id) = \dim \ker (M_S^{-1}K-\id)^*)<+\infty.
\end{equation}
Then we observe that
\begin{equation} \label{e:AsigAsig*}
A=M_S(M_S^{-1}K-\id), \qquad A^*=(M_S^{-1}K-\id)^*M_S,
\end{equation}
since $M_S$ is self-adjoint. Equalities~\eqref{e:AsigAsig*} ensure
that
$$\ker A = \ker (M_S^{-1}K-\id), \qquad M_S(\ker A^*)= \ker (M_S^{-1}K-\id)^*,$$
which, together with \eqref{e:fredhMSKs}, imply 
$$\dim \ker A = \dim \ker A^*<+\infty.$$

\subsubsection*{Step 2 -- $\boldsymbol{\ker A = \span e}$}
We already noticed that $e\in\ker A$. Let us now prove that any $y\in\ker A$ belongs to $\span e$. This property mostly follows from the strong connectivity of the graph associated to $\sigma$.

\medskip

When $X=L^2(\Om)$, we have to prove that $y$ is constant almost everywhere in $\Om$. Modifying $y$ if necessary on a zero-measure subset of $\Om$, we can choose $x_0\in\Om$ such that $y(x_0)=\esssup u$, which is finite, since $y$ lies in $L^\infty(\Om)$. It is also possible to choose $x_0$ as a Lebesgue point of $\sigma$, \ie so that $x_*\mapsto \sigma(x_0,x_*)$ is defined almost everywhere in $\Om$. Consider now $\xi\in\Om$ some Lebesgue point for both $y$ 
and $\sigma$ at the same time, for which we intend to prove that $y(\xi)=y(x_0)$. By strong connectivity, there exist pairwise distinct elements of $\Om$, $(x_1,\ldots,x_K)$, $K\in\N^*$, which are Lebesgue points for both $\sigma$ and $y$ such that $x_{K}=\xi$, and for any $0\leq k<K$, 
$$|\esupp \sigma(x_k,\cdot)|>0, \qquad x_{k+1} \in \esupp \sigma(x_k,\cdot).$$
Besides, because of \eqref{e:ptySker}, for any $0\leq k<K$, 
$$y(x_k)\int_\Om \sigma(x_k,x_*)\,\xd x_* = \int_\Om \sigma(x_k,x_*)y(x_*)\,\xd x_*.$$
For $k=0$, the previous equality implies that
$$\sigma(x_0,x_*) (y(x_0)-y(x_*))=0, \qquad \mbox{for a.e.}~x_*\in \Om,$$
and then $y(x_*)=y(x_0)$ for almost every $x_* \in \esupp \sigma(x_0,\cdot)$. In particular, it holds for $x_*=x_1$, \ie $y(x_1)=y(x_0)=\esssup u$. The conclusion $y(\xi)=y(x_{K})=\dots=y(x_0)$ is then straightforward by induction. This ensures that $y$ is equal to its essential supremum almost everywhere. 

\medskip

When $X=\R^N$, the proof follows the same idea. We provide it for completeness. Denote by $i_0$ an index such that $y_{i_0}=\max_i y_i$. Choose then an arbitrary index $i^*\neq i_0$. By strongly connectivity, there exist pairwise distinct indices $(i_k)_{1\leq k\leq K}$, $K\in\N^*$, such that $i_K=i^*$ and $\sigma_{i_ki_{k+1}}>0$ for any $k$. Consequently, for any $k$, we have
$$\sum_{j\neq i_{k+1}}  \sigma_{i_kj}(y_{j}-y_{i_k}) +
\sigma_{i_ki_{k+1}}(y_{i_{k+1}}-y_{i_k})=0, $$
which immediately implies that $y_{i_{k+1}}=y_{i_k}$ by induction. Hence, 
$y_{i_0}=y_{i^*}$ for any $i^*$, which ensures that $y=y_{i^*} e$.

\subsubsection*{Step 3 -- $\boldsymbol{\ker A = \ker A^2}$}

There is only one non-trivial inclusion, to be proved again by strong connectivity. Let $y\in \ker A^2$, so that $Ay \in \ker A$. Thanks to Step~1, there exists $\nu \in\R$ such that $Ay=\nu e$. 

When $X=L^2(\Om)$, the previous equality yields 
$$
\int_\Om \sigma(x,x_*)\,(y(x_*)-y(x))\,\xd x_* = \nu, \qquad \mbox{for a.e.}~x\in \Om.
$$ 
Modifying $y$ if necessary on a zero-measure subset of $\Om$, we can choose a Lebesgue point $x\in\Om$ for $\sigma$ such that $y(x)=\esssup y$. Hence, $\nu\leq 0$. In the same way, we get $\nu\geq 0$. Consequently, $\nu=0$ and $y\in \ker A$. 

When $X=\R^N$, we have 
$$
\sum_j \sigma_{ij}(y_j-y_i)=\nu, \qquad 1\leq i\leq N,
$$
and we successively choose $i$ as an index such that $y_i=\max_j y_j$ and $y_i=\min_j y_j$ to obtain $\nu=0$ and $y \in \ker A$. 

\subsubsection*{Step 4 -- Conclusion}

The fact that $0$ is a simple eigenvalue of $A$ is a direct consequence of Steps~2--3. The adjoint $A^*$ of $A$ naturally inherits all the properties of $A$ previously proved: $\dim \ker A^*=1$, $\ker A^* =\ker (A^*)^2$ and $0$ is a simple eigenvalue of $A^*$.
\end{proof}

\subsection{Definition of the positive weight}

Since $\ker A^*$ is one-dimensional, let us focus on a particular vector generating $\ker A^*$ and prove the first part of Theorem~\ref{t:poids}, which we recall in the proposition below and is, in some sense, at the heart of the overall proof.

\begin{proposition} \label{p:poidsrappel}
There exists a unique $v\in \ker A^*$ such that $v>0$ and $\langle v,e\rangle =1$.
\end{proposition}

\begin{proof}
The uniqueness of $v$ is straightforward, it relies on the fact that $\dim \ker A^*=1$. 
Otherwise, the proof of Proposition~\ref{p:poidsrappel} is mainly based on an homotopy argument: we use the symmetric (in fact constant) case to conclude for the non-symmetric one. Indeed, in the symmetric case, when $X=\R^N$, $v=e/N$ and, when $X=L^2(\Om)$,  $v=e/|\Om|=e$ clearly are the only elements of $\ker A^*$ satisfying the required sign and scalar-product properties. 

Denote by $M$ the constant function equal to $\|\sigma\|_{L^\infty}$ if $X=L^2(\Om)$, and the matrix with all its coefficients equal to $\max_{i,j} \sigma_{ij}$ if $X=\R^N$, and consider
the analytic function 
$$[0,1]\to W, \quad \lambda\mapsto \sigma_\lambda:=\lambda \sigma + (1-\lambda)M.$$
In the remainder of this proof, for the sake of clarity, we denote with an index
$\lambda\in[0,1]$ any matrix, operator or function built from $\sigma_\lambda$, \textit{e.g.} $A_\lambda$. Of course $A_\lambda$ inherits all the properties already known for $A$. Moreover, $\lambda\mapsto A_\lambda^*$ is analytic. 

Set $F=(\span e)^\perp=(\ker A_\lambda)^\perp=\im A_\lambda^*$, for any $\lambda\in[0,1]$. Since $\ker A_\lambda^* = \ker (A_\lambda^*)^2$, $F\cap \ker A_\lambda^*=\{0\}$. This implies that $X=F\oplus \ker A_\lambda^*$. Let $\Pi_\lambda$ be the projector onto $\ker A_\lambda^*$ along $F$, which analytically depends on $\lambda$, see \cite[Chapter~VII, \textsection~1, Section~3, Theorem~1.7]{kato}. The function of $\lambda$ defined by $v_\lambda=\Pi_\lambda e$ meets all the required properties: it is analytic, and each $v_\lambda$ is a non-trivial element of $\ker A_\lambda^*$. For the sake of completeness, we provide hereafter a more explicit construction of $\Pi_\lambda$ and $v_\lambda$. Consider the operator $J_\lambda:F\to F$, $y\mapsto A_\lambda^*y$, which is injective since $F\cap \ker A_\lambda^*=\{0\}$. Of course, $\lambda\mapsto J_\lambda$ is analytic. 

Then, for the sake of clarity, it is better to provide two different proofs depending on whether $X=L^2(\Om)$ or $X=\R^N$. 

\subsubsection*{Case 1 -- Infinite-dimensional case}

First, the operator $J_\lambda$ belongs to the Banach algebra  $\B(F)$ of bounded operators on $F$. Second, it is also surjective. Indeed, since $L^2(\Om)=\span e \oplus F$, we just have to prove that $A_\lambda^* e$ has a pre-image by $A_\lambda^*$ which lies in $F$. Let $w\neq 0$ spanning $\ker A_\lambda^*$. Using the previous direct sum, we can write $w=\alpha e +f$, for some $\alpha \in \R$ and $f\in F$. But $\alpha\neq 0$, which is proved by contradiction: if $\alpha=0$, then $\langle w,e\rangle=0$, which ensures that $e\in (\ker A_\lambda^*)^\perp=\im A_\lambda$. Since $\ker A_\lambda^2 = \ker A_\lambda$, that would imply that $e=0$, which is not the case. Consequently, we can write $e$ as $e=\frac 1\alpha(w-f)$ and then $A_\lambda^* e=-\frac 1\alpha A_\lambda^*f$. Thus $J_\lambda$ is surjective, and subsequently, bijective. 

Furthermore, thanks to the closed graph theorem, $J_\lambda^{-1}$ also lies in $\B(F)$, and since $\lambda\mapsto J_\lambda$ is clearly analytic and $\B(F)$ is a Banach algebra, then $\lambda\mapsto J_\lambda^{-1}$ is also analytic, and so is $\lambda\mapsto\Pi_\lambda=\id-J_\lambda^{-1}A_\lambda^*$. We can then conclude by setting $v_\lambda=\Pi_\lambda e$. The fact that $v_\lambda$ analytically depends on $\lambda$ is now clear. Since $e\not\in F$, $v_\lambda\neq 0$ for any $\lambda$, in $L^2(\Om)$ and almost everywhere. Finally, we have $A_\lambda^*v_\lambda=A_\lambda^*e-A_\lambda^*e=0$. 

Let us now prove that $v_\lambda$ remains positive for any $\lambda$ and almost everywhere on $\Om$. Thanks to \eqref{e:ptySker} applied to $v_\lambda$, we can write  
$$
v_\lambda(x)=\frac 1{S_\lambda(x)} \int_\Om \sigma_\lambda(x_*,x)
v_\lambda(x_*)\,\xd x_*, \qquad \mbox{for a.e.}~x\in\Om.
$$
By continuity of $v$ as a function of $\lambda\in[0,1]$, $\lambda \mapsto \essinf v_\lambda$ is also continuous. By contradiction, suppose that there exists $\mu \in(0,1]$ such that $\essinf v_\mu=0$. As in the proof of Proposition~\ref{p:poidsrappel}, modifying $v$ if necessary on a zero-measure subset of $\Om$, we can find $x\in\Om$ such that $v_\mu(x)=0$. We then intend to prove that $v_\mu=0$ almost everywhere in $\Om$, which will raise a contradiction. Consider a Lebesgue point $\xi$ for both $v_\mu$ and $\sigma_\mu$, and prove that $v_\mu(\xi)=0$. By strong connectivity, there exists $K\in\N^*$ and $x_0$, ..., $x_{K}$ Lebesgue points for $v_\mu$ and $\sigma_\mu$ satisfying $x_0=\xi$, $x_{K}=x$, and for any $0\leq k<K$, 
$$|\esupp \sigma(x_k,\cdot)|>0, \qquad x_{k+1} \in \esupp \sigma(x_k,\cdot).$$
Besides, for any $k$, we can write
$$S_\mu(x_{k+1}) v_\mu(x_{k+1})=\int_\Om\sigma_\mu(x_*,x_{k+1}) v_\mu(x_*)\,\xd x_*.$$
If we choose $k=K-1$, the left-hand side of the equality becomes $0$. Since $x_{K} \in \esupp \sigma(x_{K-1},\cdot)$, we can deduce that $v_\mu(x_{K-1})=0$. We conclude by descending induction on $k$ that $v_\mu(\xi)=v_\mu(x_0)=0$. This implies that $v_\mu=0$ almost everywhere, which is impossible: $v_\mu$ must span the one-dimensional space $\ker A_\mu^*$.

Thus, for any $\lambda$, $\essinf v_\lambda$ remains positive, which implies, in particular, that $v_1$ has the same (positive) sign as $v_0=e$ almost everywhere.

\subsubsection*{Case 2 -- Finite-dimensional case}

The proof in this case is simpler. First, the injectivity of $J_\lambda$ implies its bijectivity. Thanks, for instance, to the formula linking $J_\lambda^{-1}$ to the cofactor matrix of $J_\lambda$, $\lambda\mapsto  J_\lambda^{-1}$ is also analytic. Then $\Pi_\lambda= \id-J_\lambda^{-1} A_{\lambda}^*$ and $v_\lambda=\Pi_{\lambda} e $ inherit the required analyticity property with respect to $\lambda$. 

We prove by contradiction, as in Case~1, that all the coordinates of $v_\lambda$ are positive for any $\lambda\in [0,1]$. Assume that there exist an index $i^*$ and a real number $\mu \in (0,1]$ such that $(v_\mu)_{i^*}=0$. They can be chosen such that, for any $j$, and any $\lambda < \mu$, $(v_\lambda)j>0$, implying that $(v_\mu)_j\geq (v_\mu)_{i^*}=0$. It is possible to do so because $\lambda \mapsto (v_\lambda)_j$ is continuous on $[0,1]$ and $(v_0)_j=1/N>0$ for any $j$. 
Let $j^*\neq i^*$, by strong connectivity, there exist pairwise different indices $i_0=j^*$, $i_1$, $\ldots$, $i_r=i^*$ such that $(\sigma_\mu)_{i_ki_{k+1}}>0$ for all $k$. 

Moreover, since $A_{\mu}^* v_\mu=0$, we can write, for any $i$, 
\begin{equation} \label{e:relationker}
(v_\mu)_i=\Big({\dsp \sum_{j\neq i} (\sigma_\mu)_{ji} (v_\mu)_j}\Big)\Big{/}\Big(\sum_{j\neq i} (\sigma_\mu)_{ij}\Big). 
\end{equation}
Writing \eqref{e:relationker} for $i=i^*$, we deduce that 
$$\dsp \sum_{j\neq i^*} (\sigma_\mu)_{ji^*} (v_\mu)_j=0,$$
and it follows that $(v_\mu)_{i_{r-1}}=0$ since $(\sigma_\mu)_{i_{r-1}i^*}>0$. Then, by finite induction, applying successively \eqref{e:relationker} to $i=i_{r-1}$, $\ldots$, $i=i_{1}$, we eventually obtain $(v_\mu)_{j^*}=0$. The integer $j^*$ being arbitrary, 
it would imply that $v_\mu=0$, which is not possible. 

This ends the proof of Proposition~\ref{p:poidsrappel}. 
\end{proof}

\begin{remark}
It is interesting to provide an interpretation of $v$ in a social science
problem, for instance, when $y$ is an opinion vector of the population
regarding a binary question in a referendum. Assume that, for all $i$, $j$,
$\sigma_{ij}$ does not depend on $j$, \ie $\sigma_{ij} =\sigma_{i}$. Then
$\sigma_{i}$, which has the physical dimension of a frequency, can be seen as
the (uniform) persuasion force of the $i$-th agent on the population. 
First, we note that, for any $i$, $j$, $(A^\intercal)_{ij}= \sigma_j/N 
-\sigma_i \delta_{ij}$, where $\delta_{ij}=1$ if $i=j$ and $\delta_{ij}=0$
otherwise. Then $v$ satisfies, for any $i$, 
$$\sigma_i v_i = \frac 1N \sum_{j=1}^N \sigma_j v_j.$$
Since the null space of $A^\intercal$ is one-dimensional, the
previous equality ensures that $v$ is collinear to the vector
$(1/\sigma_1,\dots, 1/\sigma_N)^\intercal$. Therefore, for any $i$, 
the component $v_i$ of $v$ only depends on $1/\sigma_i$, and has the physical
dimension of time. More precisely, $v_i$ is proportional to the time period
in which agent $i$ interacts with the other agents in the population.
\end{remark}

\subsection{Weighted Hilbert structure on $\boldsymbol{X}$}
Proposition~\ref{p:poidsrappel} provides an element $v$ of $\ker A^*$ which satisfies $v>0$, \ie $v_i>0$ for any $i$ when $X=\R^N$ and $\essinf v >0$ if $X=L^2(\Om)$. This key property allows to build a new scalar product $\langle \cdot,\cdot\rangle_v$ and its associated norm $\|\cdot\|_v$ on $X$, weighted by $v$. The latter norm is equivalent to $\|\cdot\|$ because $v$ is lower and upper-bounded by positive constants. More precisely, when $X=\R^N$, we set
$$\langle y,z\rangle_v = \sum_{i=1}^N v_i y_i z_i, \qquad y,z\in\R^N,$$
and, when $X=L^2(\Om)$,
$$\langle y,z\rangle_v = \int_{\Om} y(x) z(x)\, v(x)\,\xd x, \qquad y,z\in L^2(\Om),$$
where $v(x)\,\xd x$ is an absolutely continuous probability measure. Note that the weighted mean can then be written in terms of weighted scalar product, \ie  
\begin{equation} \label{e:defmoypond}
\bar y^v=\langle y,v\rangle e=\langle y,e\rangle_v\, e.
\end{equation}
Noticing that $\|e\|_v=1$, $\bar y^v$ then appears as the orthogonal projection of $y$ on $\ker A$ with respect to the weighted scalar product. This leads us to identify the orthogonal complement of $\ker A$ with respect to this scalar product. In what remains, when dealing with notions related to the weighted scalar product, we shall add the index $v$, \textit{e.g.} $\perp_v$ for the orthogonality or $*_v$ for an  adjoint operator with respect to the weighted scalar product. For the sake of simplicity, we shall speak about $v$-scalar product, $v$-norm, $v$-orthogonality, $v$-adjoint.
\begin{proposition} \label{p:directsums}
The following properties hold:
\begin{enumerate}[(i)]
\item $(\ker A)^{\perp_v}=\im A$,
\item $\ker A^{*_v} = \ker A = \span e,\quad \im A^{*_v} = \im A$, 
\item $X=\ker A \oplov \im A=\ker A^* \oplor \im A$,
\item $\im A = \im A^2$,
\end{enumerate}
the direct sums in \textit{(iii)} being respectively $v$-orthogonal and orthogonal.
\end{proposition}
\begin{proof}
To obtain \textit{(i)}, let us first consider $y\in X$. Since $\langle Ay,e\rangle_v = \langle Ay,v\rangle=\langle y,A^*v\rangle=0$, then $Ay\in(\ker A)^{\perp_v}$. Conversely, if $z\in (\ker A)^{\perp_v}$, $\langle z,v\rangle=\langle z,e\rangle_v=0$, which ensures that $z\in (\ker A^*)^\perp$. Properties \textit{(ii)} are direct consequences of \textit{(i)}. The last equality in \textit{(iii)} is a well-known result which we recall here for the reader's convenience, and the previous one of course comes from \textit{(i)}. Eventually, \textit{(iv)} is straightforwardly deduced from \textit{(iii)}.
\end{proof}
\begin{remark}
Even if the previous proposition may suggest it, $A$ \textit{is not $v$-self-adjoint}. Indeed, let us consider the operator $D_v$ defined on $X$ by $D_v = \operatorname{diag} v$ if $X=\R^N$, and $D_v:z\mapsto vz$ if $X=L^2(\Om)$. We have
$$A^{*_v}={D_v}^{-1} A^* D_v.$$
If $A$ were $v$-self-adjoint, we would have $(D_vA)^*=D_vA$, which is not true in general. 
\end{remark}

\paragraph{Definition of $\boldsymbol{\pi}$ and $\boldsymbol{A_2}$.}
To conclude this section, we introduce the $v$-orthogonal projector $\pi$  on $\im A$, so that we can write $y-\bar y^v =\pi y$ for any $y\in X$, and the operator $A_2:\im A\to\im A$, $y\mapsto Ay$, which are well defined thanks to Proposition~\ref{p:directsums}. Indeed, since, for instance, $\ker A^*$ is stable by $A^*$, then $\im A= (\ker A^*)^{\perp}$ is stable by $A$. Moreover, thanks to the open mapping theorem for bounded linear operators in Banach spaces, $A_2$ is a homeomorphism of $\im A$.

\section{Proof of Theorems~\ref{t:poids}--\ref{t:convcons}} \label{s:convcons}
 
\subsection{Proof of Theorem~\ref{t:poids}}

Let us sum up the situation about the proof of Theorem~\ref{t:poids}. Proposition~\ref{p:ptykerA-kerAstar} ensures that $\ker A^*$ is a one-dimensional subspace of $X$. Then, in Proposition~\ref{p:poidsrappel}, we have proved with a homotopy argument that all elements of $\ker A^*$ have the same sign with respect to $1\leq i\le N$ or $x\in\Om$ (almost everywhere), and that there is a unique $v\in\ker A^*$ such that $v>0$ and $\langle v,e \rangle =1$. The latter property corresponds to the first part of Theorem~\ref{t:poids}. Hence, it remains to explain why the weighted mean $\bar y^v$ given in \eqref{e:defmoypond} remains constant along any solution of $\dot y(t) = Ay(t)$. In fact, it is a straightforward consequence of the definition of $v$, since
$$\frac{\xd \bar y^v}{\xd t}(t)=\langle \dot y(t),v\rangle\, e=\langle y(t),A^*v\rangle\, e = 0,$$
in finite as well as in infinite dimension. The proof of Theorem~\ref{t:poids} is then completed.

\subsection{Proof of Theorem~\ref{t:convcons} in finite dimension} \label{ss:thm2dimfinie}
The finite-dimensional case has already been studied in \cite{olsab-mur, web-the-mot}. We provide the proof for completeness and to prepare for the infinite-dimensional case. Following Proposition~\ref{p:ptykerA-kerAstar}, making a change of basis if necessary, $A$ can be written in block matrices
$$A=\begin{pmatrix} 0 & 0 \\ 0 & A_2 \end{pmatrix}.$$
We point out that 
$$y(t)-\bar y^v = e^{tA}(y(0)-\bar y^v) = e^{tA} \pi y(0) = e^{tA_2}\pi y(0), \qquad t\geq 0.$$
Besides, thanks to the Gershgorin circle theorem, as in \eqref{e:gershgorin}, we deduce that, apart from $0$, all eigenvalues of $A$ have a negative real part, \ie $A_2$ is a Hurwitz matrix. Considering that $\lambda_2$ is the second eigenvalue of $A$ (the one with the highest real part, apart from $0$), for any $\eps>0$, there exists $M(\eps)>0$ such that 
$$\| e^{tA_2} \| \leq M(\eps)\, e^{(\eps+\re \lambda_2)t}, \qquad t\geq 0.$$
The previous estimate gives the sharp exponential decay rate. 

\subsection{Proof of Theorem~\ref{t:convcons} in infinite dimension}
We already know from Proposition~\ref{p:ptykerA-kerAstar} that $0$ is a simple eigenvalue of $A$. Let us study the whole spectrum $\spec(A)$ of $A$, where the notation $\spec$ is not to be confused with the interaction function $\sigma$. We recall that $\spec(A)$ is the set of complex numbers $\lambda$ such that $A-\lambda \id$ is not bijective. 
\begin{proposition}\label{p:spectreA}
The spectrum of $A$ satisfies
$$\spec(A) \subset \{z\in \C~|~\re z \leq 0\}, \qquad \spec(A_2) =\spec(A)\backslash\{0\}\subset \{z\in \C~|~\re z <0\}.$$
\end{proposition}
\begin{proof}
Let us introduce a complex Hilbert structure on $X$, keep the same notations for the scalar products and the adjoints as in the real case, and consider the following quadratic forms respectively defined for $y\in X$ and $z\in\im A$ by
\begin{equation}\label{e:quadformC}
Q(y)=\left\langle y,\frac{A+A^{*_v}}2y\right\rangle_v = \re \left(\langle y,Ay\rangle_v\right), \qquad Q_2(z) = \re \left(\langle z,A_2z\rangle_v\right).
\end{equation}
Of course, $Q$ and $Q_2$ coincide on $\im A$. The equality between $\spec(A_2)$ and $\spec(A)\backslash\{0\}$ is then  straightforward with the $v$-orthogonal direct sum decomposition of $X$ as $\ker A \oplov \im A$. Let us now compute $Q$. We have, for any $y\in L^2(\Om;\C)$, 
$$\langle y,Ay\rangle_v = \iint_{\Om^2} v(x) \sigma(x,x_*) (y(x_*)-y(x)) \overline{y(x)}\,\xd x_* \,\xd x.$$
Applying \eqref{e:ptySker} to $v\in\ker A^*$, the previous equality can be rewritten as
$$\langle y,Ay\rangle_v = 
\iint_{\Om^2} v(x) \sigma(x,x_*) y(x_*) \overline{(y(x)-y(x_*))}\,\xd x_* \,\xd x.$$
Consequently, combining both previous expressions, we get
\begin{equation*}
Q(y)=\re \left(\langle y,Ay\rangle_v \right) = -
\frac 12 \iint_{\Om^2} v(x) \sigma(x,x_*) |y(x_*)-y(x)|^2\,\xd x_* \,\xd x,
\end{equation*}
and thus $Q$ is negative semi-definite. Let us now prove that $Q(y)=0$ if and only if $y\in \ker A$. 

If $Q(y)=0$, then $\sigma(x,x_*) |y(x_*)-y(x)|^2=0$ for almost all $x$ and $x_*$. Let $x_0$ a Lebesgue point for $y$. It is also possible to choose $x_0$ as a Lebesgue point of $\sigma$, \ie so that $x_*\mapsto \sigma(x_0,x_*)$ is defined almost everywhere in $\Om$. Consider now $X\in\Om$ some Lebesgue point for both $y$ and $\sigma$ at the same time, for which we intend to prove that $y(X)=y(x_0)$. By strong connectivity, there exist $i^*\in\N^*$ and $x_1,\ldots,x_{i^*}\in\Om$ Lebesgue points for both $\sigma$ and $u$ such that $x_{i^*}=X$, and for any $0\leq k\leq i^*-1$, 
$$|\esupp \sigma(x_k,\cdot)|>0, \qquad x_{k+1} \in \esupp \sigma(x_k,\cdot).$$
Since $\sigma(x,x_*) |y(x_*)-y(x)|^2=0$ for almost every $x$ and $x_*$, by taking $k=0$ in the previous property of $\sigma$, we get $y(x_*)=y(x_0)$ for almost every $x_* \in \esupp \sigma(x_0,\cdot)$. In particular, it holds for $x_*=x_1$, \ie  $y(x_1)=y(x_0)$. The conclusion $y(X)=y(x_{i^*})=\cdots=y(x_0)$ is then straightforward by induction. This ensures that $y$ is constant almost everywhere, \ie $y\in \span e=\ker A$.

Hence $A$ is dissipative and $A_2$ is strictly dissipative. We deduce from the Hille-Yosida theorem and one of its corollaries, see \cite[Chapter~1, Section~1.3, Theorem~3.1 \& Corollary~3.6]{pazy} for instance, that the resolvent set of $A$ contains the open right complex half-plan. This eventually leads to the required results of both spectra of $A$ and $A_2$ and ends the proof of Proposition~\ref{p:spectreA}.
\end{proof}
\begin{remark}
Of course, the result of Proposition~\ref{p:spectreA} also holds in finite dimension, and the associated quadratic form is then given by
$$Q(y) = -\frac 12 \sum_{i,j} v_i \sigma_{ij} |y_i-y_j|^2.$$
\end{remark}

In finite dimension, this is sufficient to conclude the proof of Theorem~\ref{t:convcons}, and we obtain $\rho=\re(\lambda_2)$, as noted in Section~\ref{ss:thm2dimfinie}. But, in infinite dimension, this is not enough: Proposition~\ref{p:spectreA} implies that $A_2$ is Hurwitz, but this property is however not sufficient to ensure that the corresponding semi-group is exponentially stable. We need to apply the spectral mapping theorem (see \cite{eng-nag, pazy}) and thus carefully study the spectral properties of $A_2$. Let us start with a particular easier case.

\subsubsection*{When $\boldsymbol{S}$ is constant}
Assume that $S(x)=\int_\Om \sigma(x,x_*)\,\xd x_*=\delta$ for almost every $x\in\Om$. Then $A=K-\delta \mathrm{Id}$, which ensures that $e^{tA}=e^{-\delta t}e^{tK}$.

When $\sigma$ is symmetric, the compact operator $K$ is also self-adjoint, implying that $K$ is diagonalizable with real eigenvalues. Consequently, $A_2$ is also diagonalizable, with negative eigenvalues. The convergence rate is then the Fiedler number $|\lambda_2|=-\lambda_2$.

When $\sigma$ is not symmetric, since $K$ is compact, its spectrum $\spec(K)$ contains $0$, and any element of $\spec(K)\backslash\{0\}$ is an eigenvalue with finite multiplicity. Consequently, $\spec(A)$ contains $-\delta$ and any element of $\spec(A)\backslash\{-\delta\}$ is an eigenvalue with finite multiplicity. 

\subsubsection*{General case}
We already know from Proposition~\ref{p:spectreA} that $\spec(A) \subset \{z\in \C~|~\re z \leq 0\}$. Moreover, since $A$ is a bounded operator, thanks to \cite[Corollary\,IV.1.4]{eng-nag} (for instance), $\spec(A)$ is also a compact subset of the closed disk centered at $0$ with radius $\|A\|$ in the complex plane.

Let us now use the fact that $A=K-M_S$ to study $\spec(A)$. First, since $K$ is compact, its spectrum $\spec(K)$ is countable, $0$ is the only possible accumulation point, and any nonzero element in the spectrum is an eigenvalue. Besides, from \cite[Proposition\,I.4.10]{eng-nag}, the spectrum of the multiplication operator $M_S$ is the essential range of $S$, \ie
$$
\spec(M_S)=\essran S 
= \left\{\lambda\in\C~|~ 
\operatorname{meas}(\{x\in\Om~|~|S(x)-\lambda|<\eps\})
\neq 0\,\,\,\,\forall\eps>0 \right\}.
$$
Recall that $\spec(A)$ is the disjoint union of the discrete $\spec_d(A)$ and essential $\spec_e(A)$ spectra of $A$ (see \cite{kato}). Because of the properties of $\spec(K)$, we have $\spec_e(A)=\spec_e(-M_S) \subset \essran(-S)$. Hence, $\spec_e(A)\subset\R$ and $$\sup \spec_e(A)\leq \sup \essran(-S) = -\essinf S = -\delta.$$ 
The discrete spectrum is the set of eigenvalues of $A$, but can also be seen as the subset of isolated points $\lambda$ of $\spec(A)$ such that their corresponding Riesz projector $P_\lambda$ is of finite rank. Recall (see \cite[Chapter~III, \textsection~6, Section~4, Theorem~6.17]{kato}) that
$$P_\lambda = \frac 1{2i\pi}\oint_\Gamma (z\id-A)^{-1}\xd z,$$
$\Gamma$ being a simple curve in the complex plane enclosing a region where $\lambda$ is the only element of $\spec(A)$. The Riesz projector $P_0$ associated to the eigenvalue $0$ is $\id-\pi$, which commutes with $A$. Thanks to the $v$-orthogonal direct sum decomposition from Proposition\,\ref{p:directsums}, we have
$$\spec_d(A_2) = \spec_d(A)\backslash\{0\}, \qquad \sup \essran A_2 \leq -\delta.$$
Moreover, for any $\eps\in(0,\delta)$, 
$$\spec_d(A_2)\cap \{z\in\C~|~-\delta+\eps\leq \re z<0\} = 
\spec(A_2)\cap \{z\in\C~|~-\delta+\eps\leq \re z<0\}.$$
Since $\spec(A_2)$ is compact, there is at most a finite number of elements in $\spec_d(A_2)\cap \{z\in\C~|~-\delta+\eps\leq \re z<0\}$. Hence, 
$$\sup \{\re z~|~z\in\spec_d(A_2) \}<0.$$
Consequently, the spectral bound $\mathrm{s}(A_2)=\sup \{\re z~|~z\in\spec(A_2) \}$ is negative. More precisely, if there is no eigenvalue of $A_2$ whose real part lies in $(-\delta,0)$, then $\mathrm{s}(A_2)=-\delta$. 
Otherwise, there exists $\lambda_2\in\spec_d(A_2)$ such that $\mathrm{s}(A_2)=\re \lambda_2$, $\lambda_2$ being one of the eigenvalues of $A_2$ with the highest real part. This allows to conclude on the exponential convergence towards consensus. Indeed, following \cite[Corollary\,IV.2.4]{eng-nag}, which is a consequence of the \textit{spectral mapping theorem} applied to the bounded operator $A_2$, $\mathrm{s}(A_2)$ coincides with the spectral growth of the semi-group generated by $A_2$. Therefore, $|\mathrm{s}(A_2)|$ is the convergence rate, and Theorem~\ref{t:convcons} is proved.

\section{\textcolor{black}{Further results}}\label{s:lyap}

In this section, we address three issues related to the previous analysis.
In the first subsection,
we consider a discrete-time version of \eqref{e:krause}--\eqref{e:initconddiscrete} (which is the linear version of the standard Hegselmann-Krause model \cite{heg-kra}) and discuss the convergence to consensus. In the second subsection, we clarify the relationships between the two problems studied in this article, namely \eqref{e:krause}--\eqref{e:initconddiscrete} and \eqref{e:diminfinie}--\eqref{e:diminfinieinitcond}: we prove that \eqref{e:diminfinie}--\eqref{e:diminfinieinitcond} can be obtained from \eqref{e:krause}--\eqref{e:initconddiscrete} by means of a rigorous limiting procedure as $N$ goes to $+\infty$. In the third subsection, we investigate the time asymptotics of both standard and $v$-weighted variances. In particular, we show that the weighted variance is an appropriate tool for studying the stability in the $L^2$-setting.

\subsection{Discrete-time setting}

In this subsection, we study a discrete-time version of \eqref{e:krause}--\eqref{e:initconddiscrete}, obtained by replacing time derivatives by difference quotients. Without loss of generality, we present our results in the finite-dimensional setting, but the analysis is similar in the infinite-dimensional one. 

Let $\dt>0$ and consider 
$\Gamma=(\gamma_{ij})\in\R^{N\times N}$ such that $\dt \, \Gamma$ is a stochastic 
matrix, \ie the sum of elements of each row equals $1$. We are here interested in the following problem
\begin{eqnarray}
\label{e:tempsdiscret}
y_i^{n+1} &=&\displaystyle\sum_{j=1}^N \gamma_{ij}\dt\,y_j^n, \quad 1\leq i\leq N, \quad n\in \N, \phantom{\int}\\
y_i^0&=& y_i^{\mathrm{in}},
\label{e:tempsdiscret_ic}
\end{eqnarray} 
where $y_i^n$ denotes the state variable of agent $i$ at discrete time $n\dt$. This problem is in fact the original consensus model studied in \cite{heg-kra}, but also in prior works \cite{degroot, lehrer}. The link to \eqref{e:krause}--\eqref{e:initconddiscrete} is quite clear. Indeed, we can write
$$y_i^{n+1}=\sum_{j\neq i} \gamma_{ij}\dt\,y_j^n + \gamma_{ii}\dt\,y_i^n
= \sum_{j\neq i} \gamma_{ij}\dt\,y_j^n + \left(1-\sum_{j\neq i} \gamma_{ij}\dt
\right)y_i^n,
$$
which implies
$$\frac{y_i^{n+1}-y_i^{n}}{\dt}=\sum_{j\neq i}\gamma_{ij}(y_j^n-y_i^n).$$
Hence the time-discrete problem \eqref{e:tempsdiscret} appears as the explicit 
Euler time discretization of \eqref{e:krause}. Note that the stochasticity of 
$\dt \, \Gamma$ implies that $\dt$ must satisfy the stability condition
$$\max_i \Big(\sum_{j\neq i} \gamma_{ij}\Big)\dt \leq 1.$$
For any $n\in\N^*$, we denote
$$
y^n=
\begin{pmatrix} y_1^n \\ \vdots \\ y^n_N \end{pmatrix}
\textrm{ and }
y^{\mathrm{in}}=
\begin{pmatrix} y_1^{\mathrm{in}} \\ \vdots \\ y^{\mathrm{in}}_N \end{pmatrix}.
$$
We have the following theorem , whose proof is a variant of the strategy developed in Sections~\ref{s:AAstarv} and \ref{s:convcons}, and thus is not provided. The matrix $A$ is defined by \eqref{eq:Adiscr}.

\begin{theorem} 
\label{t:poids-disc}
Assume that the graph associated to $\sigma$ is strongly connected. Then there exists a unique $v\in\ker A^*$ such that $v>0$ and $\langle v,e\rangle =1$, and the weighted mean of any solution $y^n$ to \eqref{e:tempsdiscret}--\eqref{e:tempsdiscret_ic} defined by $\bar y^v=\langle y^n,v\rangle\,e$ is constant with respect to $n$.
Moreover, there exist $\rho_*\in(0,1)$ and $M_*>0$ satisfying
$$
\|y^n-\bar y^v\| \leq M_*\, \|y^{\mathrm{in}}-\bar y^v\| \, \rho^n_*, \qquad \forall n\in\N.
$$
\end{theorem}

\subsection{Kinetic limit}

In order to study the kinetic limit of \eqref{e:krause}--\eqref{e:initconddiscrete} as $N$ goes to $+\infty$, we consider the family of (constant) functions $x_i: \ \R^+ \to \Om$, $1\leq i\leq N$, and write \eqref{e:krause}--\eqref{e:initconddiscrete} as an artificial second-order model, which is reminiscent of the classical Cucker-Smale model \textcolor{black}{\cite{cuc-sma}}: 
$$\dot x_i(t) = 0, \qquad \dot \xi_i(t) = \frac{1}{N} \sum_{j} \sigma_{ij} (\xi_j(t)-\xi_i(t)).$$
The functions $x_i$ being constant for each agent, we can re-write the interactions kernels $\sigma_{ij}$ between the agents $i$ and $j$ by introducing a continuous kernel $\sigma\ : \ \bar\Omega^2\to \R^+_*$ satisfying
$$\sigma(x_i,x_j) = N\sigma_{ij}.$$
We hence end up with
\begin{equation} \label{e:system}
\dot x_i(t)=0, \qquad \dot \xi_i(t)=\frac{1}{N} \sum_{j} \sigma(x_i,x_j) (\xi_j(t)-\xi_i(t)). 
\end{equation}
\begin{theorem}
Passing to the kinetic limit when $N$ goes to $+\infty$ gives a probability measure on $\Om^2$
$$\mu(t) = f(t,x,\xi)\, \xd x\, \xd \xi$$
solution of
\begin{equation} \label{e:chocolat}
\partial_t\mu + \operatorname{div}_\xi(X[\mu]\mu)=0
\end{equation}
or equivalently, $\partial_t f + \operatorname{div}_\xi(fX[f]) = 0$, where $\operatorname{div}_\xi$ is the divergence with respect to $\xi$, $X$ is the vector field defined by
$$X[\mu](x,\xi) = \iint_{\Om^2} \sigma(x,x_*) (\xi_*-\xi) \, \frac{1}{F(x_*)} \, \xd\mu(x_*,\xi_*),  
$$
and 
$$F(x) = \int_{\Om} \xd\mu(t,x)(\xi) = \int_\Om f(t,x,\xi)\, \xd \xi
$$
is the density marginal, which does not depend on $t$.

Moreover, the relationship between \eqref{e:chocolat} and the infinite-dimensional problem \eqref{e:pbdiminfinie} is that the function
$$
y(t,x) =\frac 1{F(x)}\int_{\Om} \xi \, f(t,x,\xi)\, \xd \xi = 
\frac{\int_{\Om} \xi \, f(t,x,\xi)\,\xd \xi}{\int_{\Om} f(t,x,\xi)\, \xd \xi}
$$
is the solution of \eqref{e:pbdiminfinie} (with the corresponding initial condition).
\end{theorem}

In the statement above, we have assumed $\mu(t)$ to be absolutely continuous to simplify the expression of the vector field $X[\mu]$. But it can be generalized without difficulty, by disintegrating $\mu(t)$ with respect to its marginal. We leave the details to the reader.

\begin{proof}
Passing to the kinetic limit can be done as in the usual Cucker-Smale model: this is done in detail in \cite[Section~2.3]{piccoli2015control2}. In a few words, one passes from the kinetic model to the finite-dimensional model by taking empirical measures
$$
\mu(t) = \frac{1}{N}\sum_{i=1}^N \delta_{(x_i(t),\xi_i(t))} .
$$ 
Conversely, the unique solution of the kinetic equation is $\mu(t) = \Phi(t)_\sharp \mu(0)$, that is the pushforward  of the initial measure under the flow $\Phi(t)$ generated by the vector field $(0,X)^\intercal$. The fact that $F$ does not depend on $t$ immediately follows by integrating \eqref{e:system} with respect to $\xi$.

Let
$$
y(t,x) = \frac 1 {F(x)}\int_{\Omega} \xi \, f(t,x,\xi)\, \xd\xi.
$$ 
We deduce that $y(t)$ is solution of \eqref{e:pbdiminfinie}. Indeed,
\begin{eqnarray*}
\partial_t y(t,x) &=& \displaystyle - \frac{1}{F(x)}\int_\Om \xi\, \operatorname{div}_\xi(f X[f])\, \xd \xi = \frac{1}{F(x)}\int_\Om f X[f]\, \xd \xi \\[10pt]
&=&  \displaystyle \frac{1}{F(x)}\int_\Om f(t,x,\xi) \iint_{\Om^2} \sigma(x,x_*) (\xi_*-\xi) \, \frac{f(t,x_*,\xi_*)}{F(x_*)} \, \xd x_*\,\xd \xi_* \\[10pt]
&=&  \displaystyle \frac{\int_\Om f(t,x,\xi)\, \xd \xi}{F(x)} \int_\Om \sigma(x,x_*) \frac{\int_\Om v_* f(t,x_*,\xi_*)\, \xd \xi_*}{F(x_*)}\, \xd x_* \\[10pt]
&&  \displaystyle \qquad\qquad\qquad - \frac{\int_\Om \xi\, f(t,x,\xi)\, \xd \xi}{F(x)} \int_\Om \sigma(x,x_*) \frac{\int_\Om f(t,x_*,\xi_*)\, \xd \xi_*}{F(x_*)} \, \xd x_* \\[10pt]
&=&  \displaystyle \int_\Om \sigma(x,x_*) ( y(t,x_*)-y(t,x) )\, \xd x_*,
\end{eqnarray*}
and the result follows.
\end{proof}

\begin{remark}
The exponential convergence of $y(t)$ to $\bar y^v$ when $t\to +\infty$ can be interpreted in the kinetic setting by saying that 
the measure $\mu(t,x)$ converges (vaguely) towards the Dirac measure $\delta_{\bar y^v}$.
\end{remark}

\subsection{Lyapunov functionals and time stabilization}

In this subsection, we present two methods to find a Lyapunov functional for our problems, allowing to recover the convergence towards consensus property with an exponential decay. The first one uses the classical Lyapunov lemma, and the second one involves a weighted variance. Eventually, we present an application of this functional to design a Jurdjevic-Quinn-type stabilizing control. It can also be used to discuss stability properties under perturbations of our system (such as nonlinearities, noises). 

\subsubsection{Using the Lyapunov lemma}

In finite dimension, since $A_2$ is a Hurwitz matrix, there exists (see, for instance, \cite{Khalil}) a unique matrix $P\in \R^{(N-1)\times (N-1)}$, symmetric positive definite, defined on $\im A$, such that
$$PA_2+A_2^\intercal P = -\I_{N-1}.$$
In infinite dimension, since $A_2$ is a homeormorphism on $\im A$ which is strictly dissipative and generates an exponentially stable semi-group, it follows from  \cite[Theorem~5.1.3]{cur-zwa} that there exists a unique bounded self-adjoint positive definite operator $P$ defined on $\im A$ such that 
$$PA_2+A_2^* P = -\id_{\im A},$$
In both cases, $P$ is given by
$$P=\int_0^{+\infty} e^{tA_2^*} e^{tA_2}\,\xd t. $$
This operator $P$ induces a norm on $\im A$, given by $\|z\|_P=\langle z,Pz\rangle=\|P^{1/2}z\|$. We define the Lyapunov functional on $\im A$, for any $z\in\im A$, as 
$$\var_P(z)=\langle z,Pz\rangle.$$
We call it the variance associated with $P$. When $X=\R^N$, we can choose $\lambda_{\max}>0$ as the highest eigenvalue of $P$. When $X=L^2(\Om)$, we notice that 
$$\|z\|^2=\langle z,z\rangle = -2 \langle z, A_2^*Pz\rangle
\leq 2\|z\| \|A_2^*\|\|P^{1/2}\|\|P^{1/2}z\|, \qquad \forall z\in\im A.$$
This implies that, in both cases, there exists $\lambda_{\max}>0$ such that $\var_P(z)\leq \lambda_{\max} \|z\|^2$ for any $z\in\im A$. 

In order to recover the exponential convergence of a solution $y$ of \eqref{e:pbdimfinie} or \eqref{e:pbdiminfinie} towards consensus, recalling that $\bar y^v \in\ker A$ and $\pi y = y-\bar y^v\in \im A$, $z=\pi y$ satisfies $\dot z = \dot y = A(z+\bar y^v) = A_2z$. Then we introduce $V_P:t\mapsto \var_P(\pi y(t))$, so that 
$$\dot V_P =\langle\dot z, Pz\rangle + \langle z, P \dot z\rangle
= \langle z, (A_2^* P + PA_2) z\rangle=-\|z\|^2\leq -\frac 1{\lambda_{\max}} V_P,$$
which ensures the required exponential convergence. This argument suffices to prove exponential convergence, but not to obtain the sharp convergence rate stated in Theorem~\ref{t:convcons}. We mention the recent  paper \cite{arn-jin-woh} for techniques to design Lyapunov functionals achieving that sharp rate.

\subsubsection{Weighted variance}

We propose an alternative variance, based on the geometric properties of our problem, and involving the weight $v$ built in Theorem~\ref{t:poids}. In the weighted scalar product framework on $X$, we define the weighted expectation
$$\exv[y] = \langle y,v\rangle=\langle y,e\rangle_v= \left\{
\begin{array}{ccl}
\displaystyle \sum_i v_i y_i & \mbox{if } & X=\R^N, \\
\displaystyle \int_\Om v(x) y(x)\,\xd x & \mbox{if } & X=L^2(\Om).
\end{array} \right. $$
It is clear that $\bar y^v=\exv[y] e$. Then we define the weighted variance of $y\in X$ as 
$$\var_v y = \exv\left[ (y-\exv[y])^2\right] = \exv[y^2]-\exv[y]^2
= \|y-\bar y^v\|_v^2=\|y\|_v^2-\|\bar y^v\|_v^2=\|\pi y\|_v^2.$$
When $X=\R^N$, we have
\begin{equation*}
\var_v y = \sum_{i} v_i (y_i-\langle y,e\rangle_v)^2
=\sum_{i} v_i y_i^2 - \langle y,e\rangle_v^2=\frac12 \sum_{i,j} v_iv_j(y_i-y_j)^2,
\end{equation*}
and when $X=L^2(\Om)$, 
\begin{multline*}
\var_v y = \int_\Om v(x) (y(x)-\bar y^v)^2\,\xd x=
\int_\Om v(x) y(x)^2\,\xd x-(\bar y^v)^2\\
=\frac12 \iint_{\Omega^2} v(x) v(x_*) (y(x)-y(x_*))^2\,\xd x_*\,\xd x.
\end{multline*}

Setting $V_v:t\mapsto \var_v(y(t))$, where $y$ solves \eqref{e:pbdimfinie} or \eqref{e:pbdiminfinie}, this weighted variance can be used as a Lyapunov functional. Indeed, let us study the monotonicity of $V_v$. Remembering that $A^{*_v}\bar y^v=0$, we have
$$\dot V_v = 2 \langle y-\bar y^v, \dot y\rangle_v
=2 \langle y-\bar y^v, Ay\rangle_v =2\langle y,Ay\rangle_v =2Q(y)=2Q_2(\pi y),$$
where $Q$ and $Q_2$ are the real Hilbert versions of the quadratic forms with the same names defined in \eqref{e:quadformC} in a complex Hilbert structure. Recall that 
$$Q(y)=\frac12\left\{ \begin{array}{ll}
\displaystyle\sum_{i,j} v_i\, \sigma_{ij} (y_j-y_i)^2 & \mbox{in finite dimension}, \\
\displaystyle\iint_{\Omega^2} v(x)\sigma(x,x_*) (y(x)-y(x_*))^2\,\xd x_*\,\xd x
& \mbox{in infinite dimension}.
\end{array}\right.$$
We also already know that $A$ and $A_2$ are generally not self-adjoint (except if $\sigma$ is symmetric), and respectively dissipative and strictly dissipative with respect to the weighted scalar product (but not to the standard scalar product). Dissipativity holds thanks to the strong connectivity assumption, as explained in the previous section. This ensures that the weighted variance strictly decreases if $y^\mathrm{in}\neq\bar y^v$. 

Then we can obtain the convergence towards to consensus thanks to the LaSalle invariance principle, see, for instance, \cite{hal69, har91, Khalil, wal80}. In finite dimension, the convergence of $V_v$ towards $0$ is then obtained because the invariant set of the differential system in the LaSalle sense is $\ker A$. Consequently, the only possible accumulation point of the trajectory is the consensus, which implies that $V_v$ converges towards $0$ when $t$ goes to $+\infty$. In infinite dimension, we also recover the convergence of $V_v$ to $0$, provided that the orbits of the system are precompact (which is true indeed, the proof is not presented here because it does not yield anything new with respect to Theorem~\ref{t:convcons}).

\medskip

\begin{remark}
The variances $\var_P$ and $\var_v$ can be respectively expressed in terms of $P$ or $D_v$, \ie for any $z\in\im A$, we have
$$\var_P(z)=\langle z,Pz\rangle, \qquad \var_v(z)=\langle z,D_v z\rangle.$$
This means that, in the same way the Lyapunov lemma provides a self-ajoint positive definite operator $P$ such that $PA_2+A_2^*P=\id_{\im A}$, we have proved here that we can find a multiplicative (diagonal) positive operator $D_v$ such that $D_v A + A^* D_v$ is dissipative (and strictly on $\im A$). 
\end{remark}

\subsubsection{Applications}

\paragraph{Jurdjevic-Quinn stabilization.}
One of the interests of Lyapunov functionals is the possibility to speed up the convergence towards consensus by adding a control (see \cite{caponigro2017mean, caponigro2017JQ}) designed by the Jurdjevic-Quinn method (see \cite{jur-qui-78}). Let $u:\R^+\to X$ be a control function, and consider the Cauchy problem
$$\dot y(t) = Ay(t) +u(t), \qquad y(0)=y^\mathrm{in}.$$
Since $\bar y^v=(\id-\pi)y$, we have
$$\frac{\xd \bar y^v}{\xd t} =(\id-\pi)Ay+(\id-\pi)u.$$
We study again $V_P:t\mapsto \langle\pi y,P\pi y\rangle$ (or $V_v$). Setting $z=y-\bar y^v$, we have
$$\dot z = Ay +u -(\id-\pi)Ay-(\id-\pi)u=\pi Ay +\pi u= A_2z+\pi u.$$
Then
$$\dot V_P =-\|z\|^2+2\langle \pi u,Pz\rangle \leq 
-\frac 1{\lambda_{\max}}V_P+2\langle \pi u,Pz\rangle.$$
We are led to choose $u=-\alpha \pi y\in \im A$, $\alpha>0$, so that 
$$\dot V_P \leq -\left(\frac 1{\lambda_{\max}}+2\alpha\right)V_P,$$
which then arbitrarily improves the convergence rate towards consensus. 

\paragraph{Robustness under a class of nonlinear perturbations.}
Another interest of the Lyapunov functionals consists in ensuring exponential convergence under some nonlinear perturbations. Let $f:X\rightarrow X$ be a function of class $C^1$, locally bounded, satisfying $\langle e,f(y) \rangle_v = \langle e,f(y) \rangle = 0$ and the dissipativity property $\langle y,f(y) \rangle_v \leq 0$ for every $y\in X$. The first property implies that $f(X)\subset\im A$; in finite dimension, it means that $\sum_i f_i(y)=0$. We consider the Cauchy problem
$$\dot y(t) = Ay(t) +f(y(t)), \qquad y(0)=y^\mathrm{in}.$$
Since $\im A=(\ker A)^{\perp_v}$, it follows from the first property of $f$ that
$$\frac{\xd \bar y^v}{\xd t} = \langle Ay+f(y),e \rangle_v~e = 0 .$$
Hence the weighted mean remains constant. Then, thanks to the dissipativity property, 
$$\dot V_v = 2 \langle y-\bar y^v, Ay+f(y)\rangle_v 
\leq 2Q_2(\pi y)-2 \langle \bar y^v, f(y)\rangle_v = 2Q_2(\pi y),$$
which implies that the above Cauchy problem is globally well posed and that the solution converges exponentially to consensus.

\section{Numerical illustrations} \label{s:num}

In this section, we present some numerical simulations in the finite-dimensional case, where the non-symmetric interaction matrix $(\sigma_{ij})$ satisfies, or not, connectivity properties. We investigate three different situations. In the first one, the population is ``fully connected'', \ie all non-diagonal coefficients of $(\sigma_{ij})$ are chosen positive. The second situation fits in the one we investigated in this work, \ie the graph associated to $(\sigma_{ij})$ is assumed to be strongly connected. We shall then refer to a ``strongly connected'' population. In our third situation, we focus on a population in which only some subgroups of the population are ``strongly connected'', the population itself being ``partially connected''. We deal with a population of $N=100$ individuals, and we focus on the collective dynamics. For the numerical simulations, we used a standard RK4 routine to solve \eqref{e:krause}--\eqref{e:initconddiscrete}.

The initial state of the population is by means of a random sampling between $0$ and $1$. 

\begin{figure}[!ht]
\begin{center}
\psfrag{State functions}{\hspace{-.2cm}{\tiny State functions}}
\psfrag{Time}{{\tiny Time}}
\psfrag{0.1}{{\tiny $\!0.1$}}
\psfrag{0.2}{{\tiny $\!0.2$}}
\psfrag{0.3}{{\tiny $\!0.3$}}
\psfrag{0.4}{{\tiny $\!0.4$}}
\psfrag{0.5}{{\tiny $\!0.5$}}
\psfrag{0.6}{{\tiny $\!0.6$}}
\psfrag{0.7}{{\tiny $\!0.7$}}
\psfrag{0.8}{{\tiny $\!0.8$}}
\psfrag{0.9}{{\tiny $\!0.9$}}
\psfrag{1}{{\tiny $\!1$}}
\psfrag{0.05}{{\tiny $0.05$}}
\psfrag{0.15}{{\tiny $0.15$}}
\psfrag{0}{{\tiny $\!0$}}
\psfrag{Logarithm of the variances}{\hspace{.75cm}{\tiny Variance}}
\psfrag{standard variance}{ {\tiny standard}}
\psfrag{weighted variance}{ {\tiny weighted}}
\includegraphics[width=7.3cm]{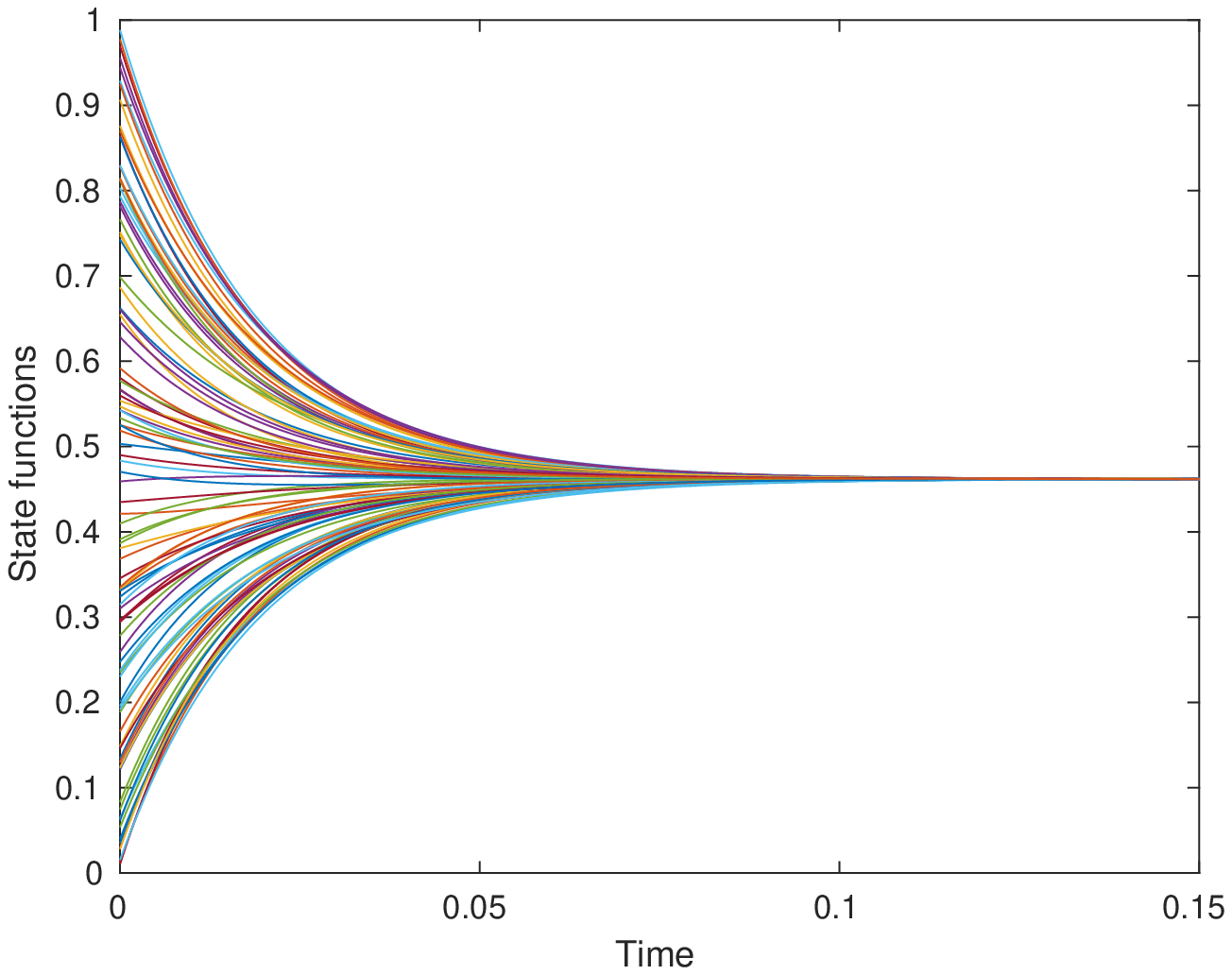} \hfill \includegraphics[width=7.3cm]{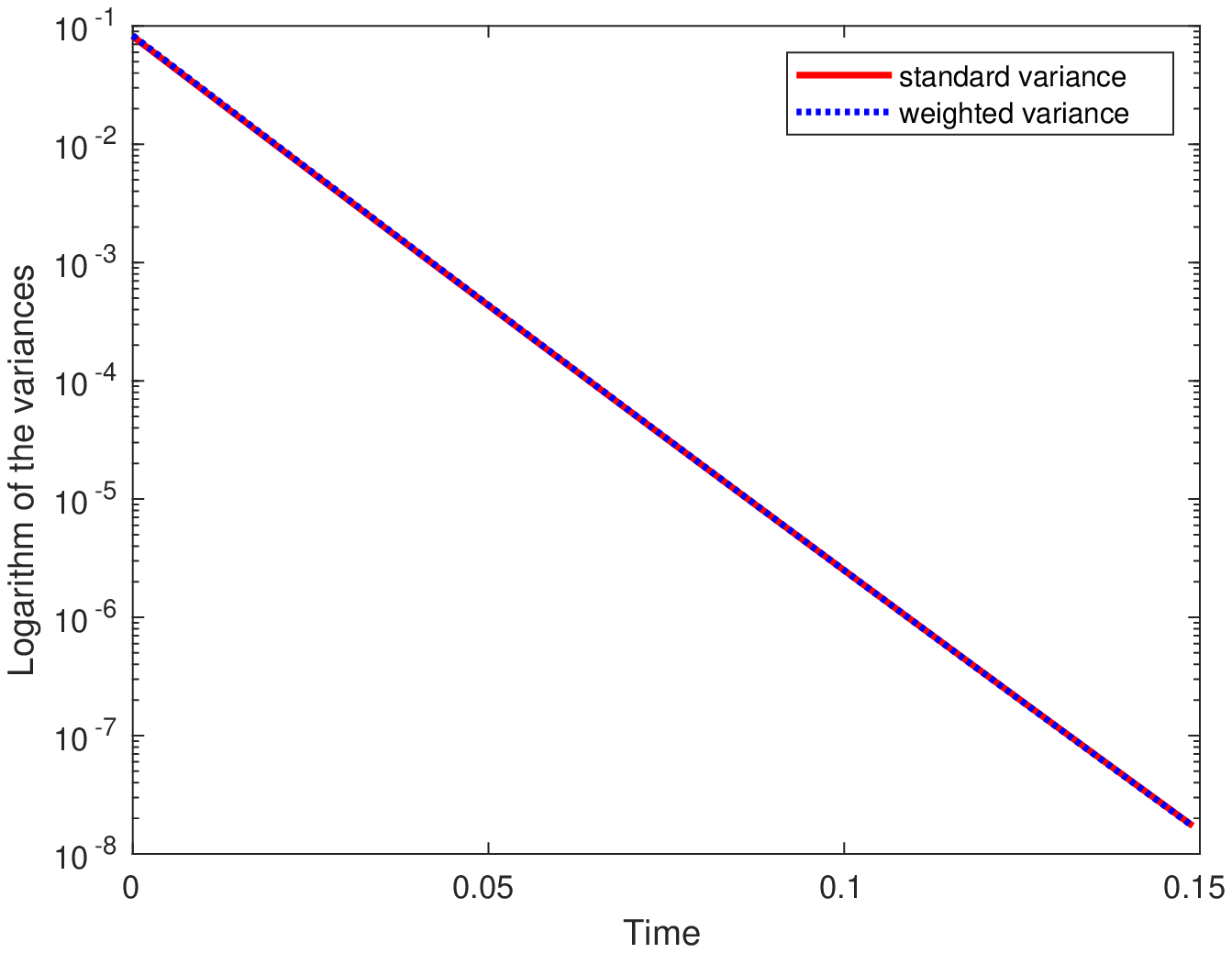}  
\caption{Time evolution on $[0,0.15]$ of (a) each $y_i$, $1\leq i\leq 100$, (b) the standard and weighted variances (log scale), for a fully connected population.} 
\label{f:Fig4}
\end{center}
\end{figure}

The first situation considers a population for which the interaction coefficients $(\sigma_{ij})$ are randomly chosen in $(0,1)$, the diagonal coefficients excepted. The computations provide a weight $v$ whose coordinates vary between $0.00830$ and $0.0118$ (to be compared to $1/N=0.01$). Figure~\ref{f:Fig4}(a) shows a very fast exponential convergence towards the weighted mean close to $0.462$. The line slope in Figure~\ref{f:Fig4}(b) is approximately $-102$. It must be compared to $2\,\,\mathrm{s}(A_2)\simeq -92.6$, since $\log V_v(t)=2\log \|y(t)-\bar y^v\|_v$. As expected, the slope is lower than $2\,\,\mathrm{s}(A_2)$, but of the same order of magnitude. 

\begin{figure}[!ht]
\begin{center}
\psfrag{State functions}{\hspace{-.2cm}{\tiny State functions}}
\psfrag{Time}{{\tiny Time}}
\psfrag{0.1}{{\tiny $\!0.1$}}
\psfrag{0.2}{{\tiny $\!0.2$}}
\psfrag{0.3}{{\tiny $\!0.3$}}
\psfrag{0.4}{{\tiny $\!0.4$}}
\psfrag{0.5}{{\tiny $\!0.5$}}
\psfrag{0.6}{{\tiny $\!0.6$}}
\psfrag{0.7}{{\tiny $\!0.7$}}
\psfrag{0.8}{{\tiny $\!0.8$}}
\psfrag{0.9}{{\tiny $\!0.9$}}
\psfrag{1}{{\tiny $\!1$}}
\psfrag{500}{{\tiny $500$}}
\psfrag{1000}{{\tiny $1000$}}
\psfrag{1500}{{\tiny $1500$}}
\psfrag{2000}{{\tiny $2000$}}
\psfrag{0}{{\tiny $\!0$}}
\psfrag{Logarithm of the variances}{{\hspace{.75cm}{\tiny Variance}}}
\psfrag{standard variance}{ {\tiny standard}}
\psfrag{weighted variance}{ {\tiny weighted}}
\includegraphics[width=7.4cm]{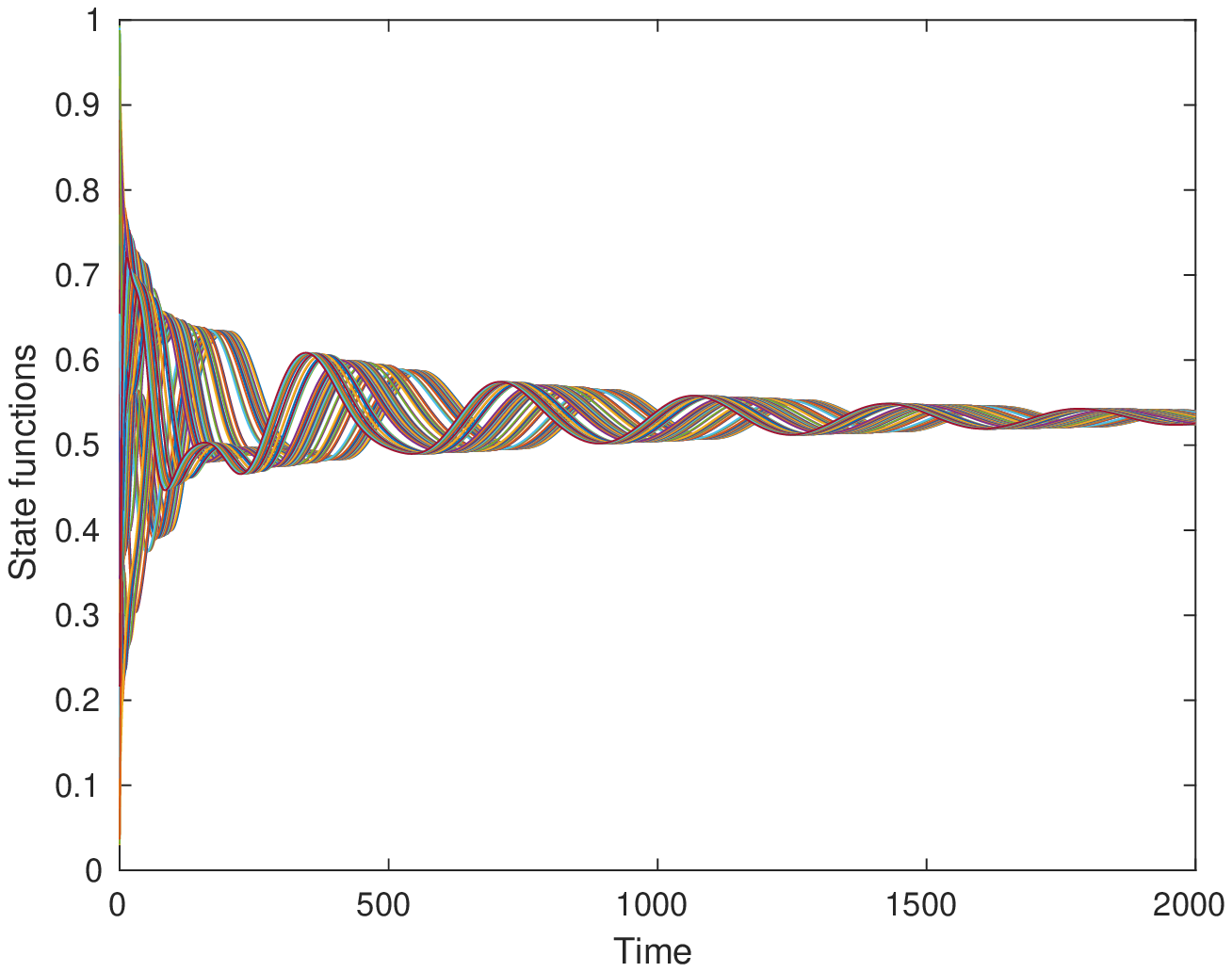} \hfill \includegraphics[width=7.4cm]{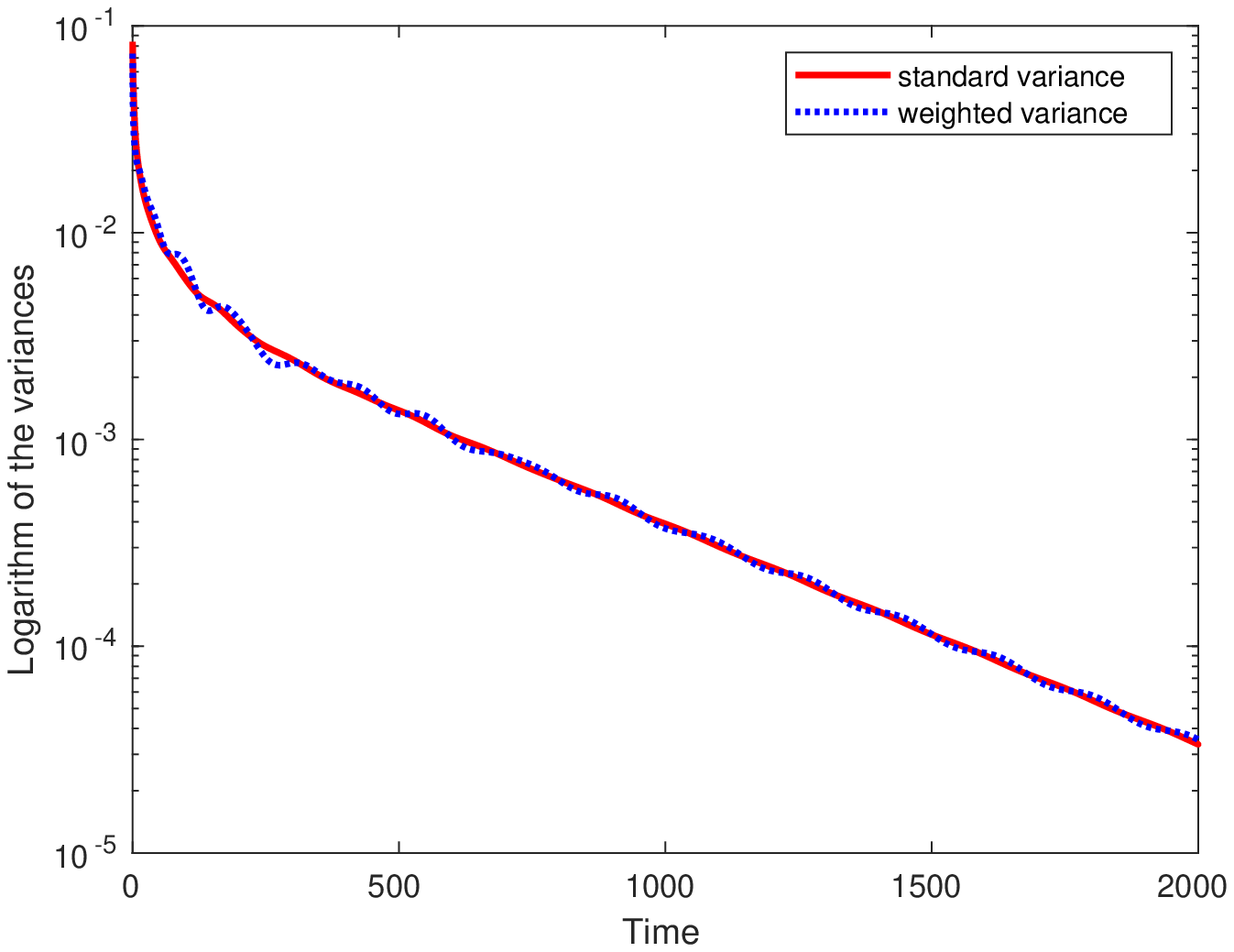}  
\caption{Time evolution on $[0,2000]$ of (a) each $y_i$, $1\leq i\leq 100$, (b) 
the standard and weighted variances (log scale), for a strongly connected population.} 
\label{f:Fig5}
\end{center}
\end{figure}

In the second situation, where the population is strongly connected, the interaction matrix is chosen such that $\sigma_{N1}$ and $\sigma_{i,i+1}$, for any $1\leq i\leq N-1$, are randomly chosen in $(0,1)$, all the other coefficients being zero. The coordinates of the weight $v$ are quite different from $0.01$ this time, they vary between $0.00291$ and $0.0530$. Figure~\ref{f:Fig5}(a) then shows a slower (with respect to the fully connected case) convergence towards the weighted mean, which is close to $0.532$. This slow convergence towards the weighted mean, with a slope equal to $-0.00240$, corresponds to the worst-case convergence scenario since $2\,\,\mathrm{s}(A_2)\simeq-0.00240$. Note that the standard and weighted variances have the same asymptotic behavior, but there are oscillations on the standard variance. It is not surprising: the weighted variance is indeed non-increasing, but we already pointed out that the standard variance does not satisfy monotonicity properties with respect to time.

\begin{figure}[!ht]
\begin{center}
\psfrag{State functions}{\hspace{-.5cm}{\small State functions}}
\psfrag{Time}{{\small Time}}
\psfrag{0.1}{{\tiny $\!0.1$}}
\psfrag{0.2}{{\tiny $\!0.2$}}
\psfrag{0.3}{{\tiny $\!0.3$}}
\psfrag{0.4}{{\tiny $\!0.4$}}
\psfrag{0.5}{{\tiny $\!0.5$}}
\psfrag{0.6}{{\tiny $\!0.6$}}
\psfrag{0.7}{{\tiny $\!0.7$}}
\psfrag{0.8}{{\tiny $\!0.8$}}
\psfrag{0.9}{{\tiny $\!0.9$}}
\psfrag{1.5}{{\tiny $\!1.5$}}
\psfrag{1}{{\tiny $\!1$}}
\psfrag{0}{{\tiny $\!0$}}
\includegraphics[width=9.3cm]{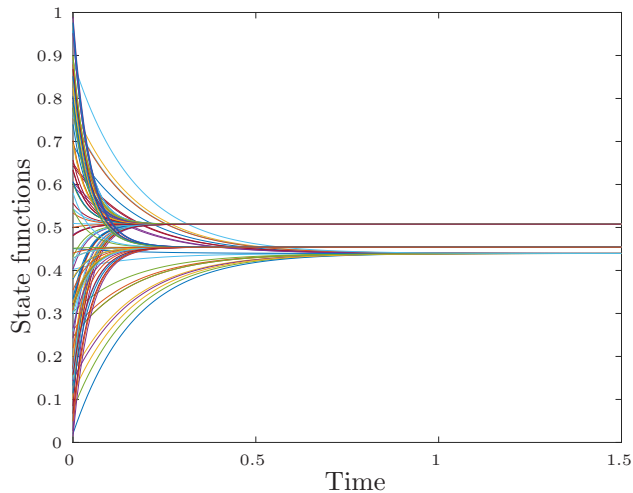} 
\caption{Time evolution on $[0,1.5]$ of each $y_i$, $1\leq i\leq 100$, in a 
partially connected population, with three strongly connected subgroups.} 
\label{f:Fig6}
\end{center}
\end{figure}

As noticed in Remark~\ref{r:partialconnect}, the third situation leads to three clusters, because the populations is divided into three non-interacting subgroups, each of them being fully connected. Hence the clustering effect appears very fast again, as it is shown on Figure~\ref{f:Fig6}. 

\section{Conclusion and prospects}\label{s:conc}

In this article, we studied the convergence to consensus, both in finite and infinite dimensions, for first-order non-symmetric systems, under the condition that the graph associated to $\sigma$ be strongly connected.

We identified a positive weight $v$ and checked that the corresponding $v$-weighted mean remains constant in time. We have moreover proved that the system exponentially converges to consensus, and exhibited the sharp exponential rate. 

The $L^2$ approach has many advantages as it allows, for instance, to use the
Jurdjevic-Quinn approach \cite{jur-qui-78} as in \cite{caponigro2017mean, caponigro2017JQ}. In our problem, we proved that the $v$-weighted variance is a Lyapunov functional, which can be an alternative to the Lyapunov functional obtained in the framework of the standard (non-weighted) $L^2$ theory.

Our analysis paves the way to further research on the subject. We mention below some of them, without claiming to be exhaustive. 

First of all, it may be interesting to study the effect of noise sources on the system, by introducing an additive noise, or by studying the system behavior when $\sigma$ is noised (\ie $A$ is noised around a fixed matrix). Another extension may consist in allowing $\sigma$ to be time-dependent. This hypothesis naturally leads to study models of influence sphere, with a possible loss of lower bounds and thus the emergence of local clusters. In such a case, the control of clusters is a question to be explored.

An open issue is the study of the system, when $\sigma$ depends on $|x_i-x_j|$, as in the original Hegselmann-Krause model \cite{heg-kra, kra}, especially the sharpness of the asymptotic convergence rate. Another open problem is the extension of our study to non-symmetric second-order models, such as generalized Cucker-Smale models \cite{caponigro2017JQ, cuc-sma, cuc-sma-2, has21, mot-tad}.

\bibliographystyle{abbrv}
\bibliography{biblio}
\addcontentsline{toc}{section}{References}

\end{document}